\RequirePackage{fix-cm}
\documentclass[smallextended]{svjour3}       
\smartqed  
\usepackage{amsfonts, amsmath}
\usepackage{graphicx}
\usepackage{epstopdf}
\usepackage{algorithm, algorithmic}
\usepackage{tikz}
\usepackage{makecell}
\usepackage{siunitx}
\usepackage{multirow}
\usepackage{arydshln}
\usepackage{hyperref}
\usepackage{caption}
\usepackage{xcolor}

\newtheorem{assumption}{Assumption}

\begin{document}

\title{NewVEM: A Newton Vertex Exchange Method for a Class of Constrained Self-Concordant Minimization Problems}

\titlerunning{A Newton Vertex Exchange Method}        

\author{Ling Liang         \and
        Kim-Chuan Toh \and \\
        Haizhao Yang
}


\institute{Ling Liang \at
              Department of Mathematics \\
              University of Tennessee, Knoxville \\
              1403 Circle Drive, Knoxville, TN 37916, USA \\
              \email{liang.ling@u.nus.edu}           
           \and
           Kim-Chuan Toh \at
           Department of Mathematics and Institute of Operations Research and Analytics \\ 
           National University of Singapore\\
           10 Lower Kent Ridge Road, Singapore 119076 \\
           \email{mattohkc@nus.edu.sg} 
           \and 
           Haizhao Yang \at 
           Department of Mathematics and Department of Computer Science\\
           University of Maryland, College Park\\
           4176 Campus Drive, College Park, MD 20742, USA \\
           \email{hzyang@umd.edu}
}

\date{Received: \today / Accepted: date}

\maketitle

\begin{abstract}
	
We propose \textbf{NewVEM}, a Newton vertex exchange method for efficiently solving self-concordant minimization problems under generalized simplex constraints. The algorithm features a two-level structure: the outer loop employs a projected Newton method, and the inner loop uses a vertex exchange approach to solve strongly convex quadratic subproblems. Both loops converge locally at linear rates under technical conditions, resulting in a ``fast $\times$ fast'' framework that demonstrates high efficiency and scalability in practice. To get a feasible initial point to execute the algorithm, we also present and analyze a highly efficient semismooth Newton method for computing the projection onto the generalized simplex. The excellent practical performance of the proposed algorithms is demonstrated by a set of numerical experiments. Our results further motivate the potential real-world applications of the considered model and the proposed algorithms.

\end{abstract}

\section{Introduction}

A function $f:\mathbb{R}^n\to \mathbb{R}\cup \{+\infty\}$ is said to be self-concordant with parameter $M \geq 0$ if it is three-time continuously differentiable and the function $\phi(\tau):=f(x+\tau v)$ satisfies $|\phi'''(\tau)|\leq M\phi''(\tau)^{\frac{3}{2}}$ for any $\tau\in \mathrm{dom}(\phi)$, $x\in \mathrm{dom}(f)$, and $v\in\mathbb{R}^n$. Particularly, if $M=2$, we say that $f$ is standard self-concordant. Self-concordant minimization has attracted increasing attention in recent years because it ensures strong theoretical guarantees, enables robust practical algorithms, and forms the backbone of efficient second-order methods for a wide range of convex optimizations~\cite{boyd2004convex}. Reader are referred to~\cite{tran2015composite,sun2019generalized,liu2022newton} for a more detailed description of the class of self-concordant functions and many interesting examples arising from practical applications, including optimal experimental design~\cite{pukelsheim2006optimal,damla2008linear,lu2013computing,hendrych2023solving}, quantum tomography~\cite{gross2010quantum}, logistic regression~\cite{tran2022new}, portfolio optimization~\cite{ryu2014stochastic}, and optimal transport~\cite{peyre2019computational}.

Mathematically, we aim to solve the following generalized simplex constrained self-concordant minimization problem:
\begin{equation}
	\label{eq-sc-min}
	\min_{x\in \mathbb{R}^n}\; f(x)\quad \mathrm{s.t.}\quad e^Tx = b,\; \ell\leq x\leq u,
	\tag{SC}
\end{equation}
where the objective function $f$ is self-concordant, $e\in \mathbb{R}^n$ denotes the vector of all ones, and $b\in\mathbb{R}$, $\ell, u\in\mathbb{R}^n$ satisfying $-\infty < \ell \leq u < \infty$. Obviously, when $b = 1$, $\ell = 0$ and $u = e$, the feasible set of problem~\eqref{eq-sc-min} reduces to the standard simplex. In view of this, we call the feasible set of problem~\eqref{eq-sc-min}, denoted as $\mathcal{F}:=\{x\in\mathbb{R}^n\;:\; e^Tx = b, \;\ell\leq x\leq u\}$, as the {\em generalized simplex} for the rest of this paper. For notational simplicity, we denote $\mathcal{C}(\ell, u):=\{x\in\mathbb{R}^n\;:\;  \ell \leq x \leq u\}$ and assume that it is nonempty. It is clear that $\mathcal{F}$ is nonempty if and only if $e^T\ell \leq b \leq e^Tu$. Moreover, if either $e^T\ell =b$ or $e^Tu = b$ holds, then we see that the feasible set of is a singleton (i.e., $\{\ell\}$ or $\{u\}$) and there is no need to solve the problem. Also, if there exists an index $i\in[n]$ such that $\ell_i = u_i$, then we see that $x_i = \ell_i = u_i$ is fixed. In this case, we can fix this variable and consider a smaller-scale problem. Hence, without loss of generality, we make the following assumption for the rest of this paper. 
\begin{assumption}
	\label{assumption-feasible-set}
	The solution set of problem~\eqref{eq-sc-min} is nonempty, the Hessian $\nabla^2f(x)$ is nondegenerate for any $x\in \mathrm{dom}(f)$, and it holds that $e^T\ell < b < e^Tu$ and $-\infty < \ell < u < \infty$.
\end{assumption}
Under Assumption \ref{assumption-feasible-set}, we can also see that the Slater's condition~\cite{slater2013lagrange} holds, and for any feasible solution $x$, both sets $\{i\in[n]\;:\; x_i > \ell_i\}$ and $\{i\in[n]\;:\; x_i < u_i\}$ are nonempty. Moreover, the set $\mathcal{F}$ is nonempty, convex and compact. We shall mention here that when the feasible set is given by $a^Tx = b,\; \ell\leq x\leq u$ where all entries of $a\in\mathbb{R}^n$ are nonzero, by a change of decision variable $x' := \mathrm{Diag}(a) x$, we get a problem that is of the same form as~\eqref{eq-sc-min}.

Though the constraint set of problem~\eqref{eq-sc-min} is relatively simple, it admits strong modeling power due to the fact that many important problems can be modeled by using the generalized simplex constraint, including support vector machine~\cite{chen2024study}, portfolio optimization~\cite{shepp1982maximum,algoet1988asymptotic}, positron emission tomography~\cite{ben2001ordered,cornuejols2006optimization}, and optimal experimental design~\cite{bohning1986vertex,pukelsheim2006optimal,damla2008linear,lu2013computing,yang2020efficient,hendrych2023solving}, to mention just a few. Motivated by the essential roles of the self-concordant functions and the generalized simplex in numerous practical applications, our primary goal in this work is to develop an efficient, scalable and reliable solution method for solving the problem~\eqref{eq-sc-min}.  

As a convex optimization problem, \eqref{eq-sc-min} can be solved using interior point methods (IPMs) via conic programming or disciplined convex programming~\cite{nesterov1994interior,grant2006disciplined,lu2013computing}. However, IPMs often struggle with scalability due to its high computational cost per iteration, making them less suitable for large-scale problems. To address this, first-order methods have gained significant attention.

For the D-optimal experimental design problem (see Section \ref{subsec-eg}) with the standard simplex constraint, \cite{bohning1986vertex} proposed a vertex exchange method with guaranteed global convergence. However, its (local) linear convergence 
was not explored. In practice, the method typically requires a large number of iterations to achieve a desired accuracy, which is computationally expensive given the 
time-intensive nature of evaluating the objective function and its gradient. Furthermore, the convergence properties in~\cite{bohning1986vertex} do not extend to generalized simplex constraints, to the best of our knowledge~\cite{ahipacsaouglu2021branch}. Consequently, whether the vertex exchange method converges for problem \eqref{eq-sc-min} remains unclear.

The classical projected gradient method is another option, offering local linear convergence~\cite{tran2015composite}. However, it typically requires many iterations to reach the region where linear convergence applies. Additionally, each iteration involves carefully selecting a proximal metric (and step size) and performing projections onto the feasible set, which adds to the computational cost. These challenges motivate us to develop a line-search-free algorithmic framework coupled with an efficient routine for projecting onto the feasible set.

The Frank-Wolfe method (a.k.a the conditional gradient method) is another popular first-order approach that avoids projections with global convergence guarantees~\cite{frank1956algorithm}. Recent advances show that the away-step Frank-Wolfe method achieves affine-invariant and norm-independent global linear convergence rates for both the objective gap and the Frank-Wolfe gap~\cite{zhao2023analysis,zhao2023away}. However, it requires a line search at each iteration and is applicable only when the vertices of the feasible set are explicitly known (e.g., for the standard simplex). For general polyhedra, identifying the vertices is challenging~\cite{khachiyan2009generating}, limiting the practicality of Frank-Wolfe-type methods in real-world applications.

To address these issues, \cite{liu2022newton} proposed an inexact projected Newton method for solving \eqref{eq-sc-min}, which achieves local linear convergence. A key advantage of this approach is that, after a few iterations, the iterates are guaranteed to enter the region where linear convergence applies. This significantly reduces the total number of iterations, making it particularly beneficial when evaluating the objective function and its derivatives is computationally expensive. Each iteration of the inexact projected Newton method requires solving a strongly convex quadratic programming (QP) subproblem, which is handled using Frank-Wolfe methods in~\cite{liu2022newton}. When the feasible set is the standard simplex, the away-step Frank-Wolfe method ensures global linear convergence~\cite{lacoste2015global}. However, for general polyhedral sets, where identifying all vertices is computationally challenging, the away-step Frank-Wolfe method becomes less practical.

In this work, we propose solving the QP subproblems in the inexact projected Newton method using the vertex exchange method. Notably, the vertex exchange method not only guarantees global convergence but also enjoys a local linear convergence rate. This leads to a robust double-loop framework with strong convergence properties, where both the outer and inner loops are highly efficient, resulting in a ``fast $\times$ fast'' algorithm. Since the proposed algorithm requires a feasible starting point, we also develop and analyze an efficient algorithm for projecting onto the generalized simplex, ensuring a computationally inexpensive initialization. To evaluate the effectiveness of our approach, we conduct an extensive numerical study. Our results indicate that the proposed algorithm is not only easy to implement but also achieves superior practical performance compared to state-of-the-art solvers, highlighting its potential for large-scale constrained self-concordant minimization problems.

\subsection{Motivating example}\label{subsec-eg}

Next, we present a motivating example, demonstrating the critical role of solving the problem~\eqref{eq-sc-min} in the context of exact optimal experimental design~\cite{atkinson1993optimum}. Specifically, consider the problem
\begin{equation}\label{eq-experimental-design}
	\min_{x \in \mathbb{R}^n}\; f(x)
	\quad 
	\text{s.t.}
	\quad
	e^T x = b,
	\quad 
	x \in \mathbb{Z}_+^n,
\end{equation}
where $f(x)$ is the objective function, $b \in \mathbb{Z}_{++}$ is the given experimental budget, and the decision variable $x$ must take nonnegative integral values. In essence, the goal is to select exactly $b$ experiments so as to optimize $f(x)$. We assume that $f$ is self-concordant as defined earlier. Two well-known objective functions in this setting are:
\[
f(x) \;=\; \log \operatorname{tr}\bigl(A\,\mathrm{Diag}(x)\,A^T\bigr)^{-1}
\quad \text{and} \quad
f(x) \;=\; -\log \det\bigl(A\,\mathrm{Diag}(x)\,A^T\bigr),
\]
where $
A 
\;=\; 
\begin{bmatrix}
	a_1 & \dots & a_n
\end{bmatrix} 
\;\in\; \mathbb{R}^{p \times n}
$
represents the design space. These objective functions correspond to the famous A-optimal and D-optimal experimental designs, respectively, and both are self-concordant~\cite{hendrych2023solving}.

Due to the integral constraint $x \in \mathbb{Z}_+^n$, solving problem \eqref{eq-experimental-design} is generally NP-hard, making it extremely challenging to solve in practice. A commonly used approach for solving mixed-integer nonlinear programming (including problem~\eqref{eq-experimental-design}) is the branch-and-bound (BnB) method~\cite{lawler1966branch}. A key computational bottleneck of the BnB method is the need to solve a large number of optimization problems of a similar form as problem~\eqref{eq-sc-min}, where $\ell$ and $u$ vary as the BnB method iterates. Naturally, one can apply existing methods (e.g., interior-point methods, augmented Lagrangian methods, and operator splitting methods) to solve these subproblems. However, these techniques typically require solving linear systems at each iteration, which may render them unsuitable in the BnB framework due to their high per-iteration computational cost. In contrast, the inexact projected Newton method described in Algorithm~\ref{alg-pn} is a promising alternative.

\subsection{Outline}
The rest of this paper is organized as follows. We first present some notation and background information that is needed for our later exposition in Section \ref{section-preliminary}. Then in Section \ref{section-proj-newton}, we present the inexact projected Newton method for solving the problem \eqref{eq-qp} and its convergence properties. Next, we also present and analyze a fast semismooth Newton for computing the projection on the generalized simplex in Section~\eqref{section-proj-F}. To solve the quadratic subproblem at each iteration of the inexact projected Newton method, we present the vertex exchange method together with its convergence properties in Section~\ref{section-vem-qp}. Moreover, to evaluate the practical performance of the proposed method, we conduct a set of extensive numerical experiments in Section \ref{section-exp}. Finally, we provide some concluding remarks and future research directions in Section \ref{section-conclusions}.

\section{Preliminary} \label{section-preliminary}

We first introduce some notation that will be used throughout this paper. Let $x\in\mathbb{R}^n$ be a column vector and $i\in[n]:=\left\{1,\dots,n\right\}$, we use $x_i$ to denote the $i$-th entry of $x$. For a given matrix $A\in \mathbb{R}^{n\times n}$, the $(i, j)$-entry of $A$ is denoted as $A_{ij}$ and the $i$-th column of $A$ is denoted as $A_{:, i}$, for $1\leq i , j\leq n$. Moreover, the vector formed by all the diagonal entries of $A$ is denoted as $\mathrm{diag}(A)$. Similarly, for any vector $x\in\mathbb{R}^n$, the diagonal matrix in $\mathbb{R}^{n\times n}$ whose diagonal entries are those of $x$ is denoted as $\mathrm{Diag}(x)$.

Let $H\in\mathbb{S}^n_{++}$, we use $\left\langle \cdot, \cdot\right\rangle_H$ to denote the inner product with respect to $H$, i.e., $\left\langle x, y\right\rangle_H = x^THy$. Correspondingly, the induced weighted norm is denoted as $\|\cdot\|_H$. Conventionally, when $H$ is the identity matrix, we use $\left\langle\cdot,\cdot\right\rangle$ and $\|\cdot\|$ to denote the Euclidean inner product and Euclidean norm, respectively. 

Let $f:\mathbb{R}^n\to \mathbb{R}\cup \{\pm\infty\}$ be an extended real-valued function, we denote the effective domain of $f$ as $\mathrm{dom}(f):=\{x\in\mathbb{R}^n\;:\; f(x) <+\infty\}$. Let $f:\mathbb{R}^n\to \mathbb{R}\cup\{\pm \infty\}$ be an extended-valued convex function, we denote the convex conjugate function of $f$ as $f^*(z) := \sup\{\langle x, z\rangle - f(x)\; :\; x\in\mathbb{R}^n\}$, and the subdifferential of $f$ at a point $x\in\mathbb{R}^n$ as $\partial f(x):=\{v\in\mathbb{R}^n\;:\; f(x')\geq f(x) + v^T(x'-x),\;\forall x'\in \mathbb{R}^n\}$. 

For a given closed convex and nonempty subset $\mathcal{C}$ of $\mathbb{R}^n$, we use $\mathcal{I}_{\mathcal{C}}(\cdot)$ to denote the indicator function of $\mathcal{C}$, which is a closed proper convex function. Thus, the convex conjugate function of $\mathcal{I}_{\mathcal{C}}(\cdot)$ is denoted as $\mathcal{I}_{\mathcal{C}}^*(\cdot)$, which is also known as the support function of the set $\mathcal{C}$. 
Note that $\partial \mathcal{I}_{\mathcal{C}}(x)$ coincides with the indicator function of the normal cone of $\mathcal{C}$ at the reference point $x$ (i.e., $\mathcal{N}_{\mathcal{{C}}}(x)$). In particular, we have that $\mathcal{{N}_{\mathcal{C}}}(x):=\{v\in \mathbb{R}^n\;:\; v^T(x'-x)\leq 0,\; \forall x'\in \mathcal{C}\}$. Moreover, the projection onto the set $\mathcal{C}$ is denoted as $\mathrm{Proj}_{\mathcal{{C}}}(\cdot)$. When $\mathcal{C} = \mathcal{C}(\ell, u)$, it is known that  the support function of $\mathcal{C}(\ell, u)$ is given as (see e.g., \cite[Section 6.1]{banjac2019infeasibility})
\[
    \mathcal{I}_{ \mathcal{C}(\ell, u)}^*(z) = \ell^T\min\{z,0\} + u^T\max\{z, 0\},\quad \forall z\in\mathbb{R}^n,
\]
where ``min'' and ``max'' are applied element-wise. Moreover, by the definition of the normal cone, one can verify that for any $x\in \mathcal{C}(\ell, u)$, $v\in \mathcal{N}_{\mathcal{C}(\ell, u)}(x)$ admits the following expression (see e.g., \cite[Example 6.10]{rockafellar2009variational})
\begin{equation}
    \label{eq-normalcone-C}
    v_i \left\{
        \begin{array}{ll}
             = 0, & \ell_i < x_i < u_i, \\
             \leq 0, & x_i = \ell_i, \\
             \geq 0, & x_i = u_i,
        \end{array}
    \right.\quad \forall \;i \in [n],
\end{equation}
provided that $\ell < u$. We refer the reader to \cite{rockafellar1997convex,rockafellar2009variational} for more comprehensive studies of convex analysis and variational analysis.

Next, we present the definition of (strong) semismoothness of a nonsmooth mapping that is needed in analyzing the convergence of a certain Newton-type method; see \cite{mifflin1977semismooth,qi1993nonsmooth,sun2002semismooth,facchinei2003finite} and references therein for more details on semismoothness.

\begin{definition}
    \label{def-semi-smooth}
    Let $\mathbb{X}$ and $\mathbb{Y}$ be two finite dimensional Euclidean spaces. Let $\mathcal{M}:\mathcal{O}\subseteq \mathbb{X}\to \mathbb{Y}$ be a locally Lipschitz continuous function on the open set $\mathcal{O}$ \footnote{Hence, the Jacobi of $\mathcal{M}$, namely $\mathcal{M}'$, is well-defined almost everywhere in the sense of Lebesgue measure on $\mathcal{O}$ \cite{rademacher1919partielle}.} whose generalized Jacobian \cite{clarke1990optimization} is defined as $\partial \mathcal{M}(x):= \mathrm{conv}(\partial_B \mathcal{M}(x))$, where $\partial_B \mathcal{M}(x)$ denotes the B-subdifferential \cite{qi1993convergence} of $\mathcal{M}$ at $x\in\mathcal{O}$ that is defined as
    \[
        \partial_B \mathcal{M}(x):=\left\{\lim_{k\to\infty}\mathcal{M}'(x^k)\;:\; \lim_{k\to\infty}x^k = x,\; \mathcal{M}'(x^k) \textrm{ exists}\right\},
    \]
    and ``\textrm{conv}'' stands for the convex hull in the usual sense of convex analysis \cite{rockafellar1997convex}. Then, $\mathcal{M}$ is said to be semismooth at $x\in\mathcal{O}$ if $\mathcal{M}$ is directionally differentiable at $x$ and for any $\mathcal{V}\in\partial \mathcal{M}(x + \Delta x)$ with $\Delta x\to 0$, 
    \[
        \mathcal{M}(x + \Delta x) - \mathcal{M}(x) - \mathcal{V}\Delta x = o(\|\Delta x\|).
    \]
    Moreover, $\mathcal{M}$ is said to be strongly semismooth at $x\in\mathcal{O}$
    if the error term $o(\|\Delta x\|)$ above is replaced by $O(\|\Delta x\|^2).$
    Finally, $\mathcal{M}$ is said to be semismooth (respectively, strongly semismooth) on $\mathcal{O}$ if $\mathcal{M}$ is semismooth (respectively, strongly semismooth) at every point $x\in\mathcal{O}$.
\end{definition}

Roughly speaking, (strongly) semismooth mappings are locally Lipschitz continuous mappings for which the generalized Jacobians define legitimate approximation schemes, which turn out to be broad enough to include most of the interesting cases. Particularly, the following result is well-known. 
\begin{lemma}[{\cite[Propositions 7.4.4 \& 7.4.7]{facchinei2003finite}}]
    \label{lemma-semismooth-piecewise-linear}
    Every piecewise affine mapping from a finite dimensional Euclidean space to another is strongly semismooth everywhere. Moreover, the composition of strongly semismooth mappings is still strongly semismooth.
\end{lemma}

\section{An inexact projected Newton method}\label{section-proj-newton}

In this section, we describe the inexact projected Newton method of~\cite{liu2022newton} for solving problem~\eqref{eq-sc-min} and present its strong convergence properties. 

The following notation will be useful in our later exposition. Let $x\in \mathrm{dom}(f)$ and assume that $\mathrm{dom}(f)$ does not contain straight lines, it holds that $\nabla^2 f(x)$ is positive definite. In this case, we denote the local norm associated with $f$ together with its dual norm as $\|\cdot\|_x:=\|\cdot\|_{\nabla^2 f(x)}$ and $\|\cdot\|_x^*:=\|\cdot\|_{(\nabla^2 f(x))^{-1}}$, respectively, for notational simplicity. The following univariate function $h:\mathbb{R}_+\to \mathbb{R}_+$ defined as follows is also useful:
\[
h(\tau):=\frac{\tau(1-2\tau + 2\tau^2)}{(1-2\tau)(1-\tau)^2 - \tau^2},\quad \tau\geq 0.
\]
It is easy to see that $h$ is monotonically increasing and hence its inverse function $h^{-1}(\tau)$ is well-define for $\tau > 0$. Using these notation, the inexact projected Newton method is presented in Algorithm \ref{alg-pn}.

\begin{algorithm}[htb!]
	\caption{An inexact projected Newton method for constrained self-concordant minimization.}
	\label{alg-pn}
	\begin{algorithmic}[1]
		\STATE \textbf{Input:} $x^0\in \mathrm{dom}(f)\cap \mathcal{F}$, parameters $\beta\in (0, 1/20),\; \sigma\in (0,1), C > 1$ such that $1/(C(1-\beta)) + \beta/((1-2\beta)(1-\beta)^2) \leq \sigma$ and $1/C + 1/(1-2\beta) \leq 2$, $C_1\in (0, 1/2)$ and $\delta\in (0,1)$.
		\STATE Set $\lambda^{-1} = \frac{\beta}{\sigma}$ and $\xi^0 = \min\{\beta/C, C_1h^{-1}(\beta)\}$.
		\FOR{$k\geq 0$}
		\STATE Compute an approximate solution $\tilde{x}^{k+1}$ of the following QP
		{by using Algorithm \ref{alg-qp} with $\tilde{x}^k$ as the initial point:} 
		\[
		\min_{x\in\mathbb{R}^n}\; q_k(x):=\frac{1}{2}(x-x^k)^T\nabla^2f(x^k)(x-x^k) + \nabla f(x^k)^T(x-x^k)\quad \mathrm{s.t.}\quad x\in\mathcal{F},
		\]
		in the sense that 
		\begin{equation}\label{eq-alg-pn-inexactness}
			\max_{x\in\mathcal{F}}\; \nabla q_k(\tilde{x}^{k+1})^T(\tilde{x}^{k+1} - x) \leq \xi^k.
		\end{equation}
		
		\STATE Set $\Delta x^k = \tilde{x}^{k+1} - x^k$ and compute $\gamma^k = \|\Delta x^k\|_{x^k}$.
		
		\IF{$\gamma^k+\xi^k\leq h^{-1}(\beta)$ or $\lambda^{k-1}\leq \beta$}
		\STATE Set $\lambda^k = \sigma \lambda^{k-1}$, $\xi^{k+1} = \sigma \xi^k$, and $x^{k+1} = \tilde{x}^{k+1}$.
		\ELSE 
		\STATE Set $\lambda^k = \lambda^{k-1}$, $\xi^{k+1} =  \xi^k$, and $x^{k+1} = x^k + \frac{\delta ((\gamma^k)^2 - (\xi^k)^2)}{(\gamma^k)^3+ (\gamma^k)^2 - (\xi^k)^2\gamma^k}\Delta x^k$.
		\ENDIF
		\IF{$\lambda^k = 0$}
		\STATE \textbf{Output:} $x^{k+1}$.
		\ENDIF
		\ENDFOR 
	\end{algorithmic}
\end{algorithm}

The convergence properties of Algorithm \ref{alg-pn} are summarized as follows; see \cite{liu2022newton} for the detailed analysis.

\begin{theorem}
	\label{thm-convergence-pn}
	Let $\omega(\tau):=\tau - \log(1+\tau)$ and set 
	\[
	K:= \frac{f(x^0) - f(x^*)}{\delta\omega(h^{-1}(\beta)(1-2C_1)/(1-C_1))},
	\]
	where $x^*$ is an optimal solution of problem \eqref{eq-sc-min}. Then, there exists at least one $k\in [K]$ such that $\gamma^k+\xi^k\leq h^{-1}(\beta)$. Moreover, for $k\geq K$, it holds that 
	\[
	\|x^k - x^*\|_{x^*} = O(\sigma^k).
	\]
\end{theorem}

We see that after a finite number of iterations $K$, Algorithm~\ref{alg-pn} is guaranteed to take the full step, which leads to a linear convergence rate. Our numerical experience indicates that $K$ is typical very small, hence, the total number of iterations needed for Algorithm~\ref{alg-pn} to produce a highly accurate solution is also small. This makes Algorithm~\ref{alg-pn} a strong candidate when compared with the projected gradient method and the Frank-Wolfe method.

\section{Fast projection onto $\mathcal{F}$} \label{section-proj-F} 

Recall that the feasible set of problem \eqref{eq-qp} is denoted as 
\[
\mathcal{F}:=\{x\in\mathbb{R}^n\;:\; e^Tx = b, \;\ell\leq x\leq u\},
\]
where $e^T\ell < b < e^Tu$ and $-\infty < \ell < u < \infty$ under Assumption \ref{assumption-feasible-set}. Our goal in this subsection is to propose an efficient algorithm for computing the projection onto the set $\mathcal{F}$ to high precision (e.g., machine precision), which is needed at the beginning of Algorithm \ref{alg-qp}. Mathematically speaking, we aim at solving the following special form of a strongly convex QP problem:
\begin{equation}
	\label{eq-proj-F}
	\min_{x\in\mathbb{R}^n}\; \frac{1}{2}\|x - \bar{x}\|_2^2\quad \mathrm{s.t.}\quad  x\in\mathcal{F},
\end{equation}
where $\bar{x} \in\mathbb{R}^n$ is the given point to be projected. It is clear that problem \eqref{eq-proj-F} is a special case of problem \eqref{eq-qp} with $Q := I_n$ and $c := -\bar{x}$, respectively. According to \eqref{eq-qp-dual}, the Lagrange dual problem (reformulated as an equivalent minimization problem) of problem \eqref{eq-proj-F} reads
as follows:
\begin{equation}
	\label{eq-proj-F-dual}
	\min_{y\in\mathbb{R},\; w\in \mathbb{R}^n,\; z\in\mathbb{R}^n}\; -by + \frac{1}{2}\|w\|^2 + \mathcal{I}_{\mathcal{C}(\ell, u)}^*(-z)\quad \mathrm{s.t.}\quad -w + ye + z = -\bar{x}, 
\end{equation}
and the associated KKT conditions are given by 
\begin{equation}
	\label{eq-kkt-proj-F}
	e^Tx = b,\quad x\in \mathcal{C}(\ell, u),\quad -w + ye + z = -\bar{x},\quad x-w = 0,\quad -z\in \mathcal{N}_{\mathcal{C}(\ell, u)}(x).
\end{equation}
Recall that the Slater condition implies that the above system of KKT conditions \eqref{eq-kkt-proj-F} admits at least one solution \cite[Theorem 2.165]{bonnans2013perturbation}. Note also that the conditions $x\in\mathcal{C}(\ell, u)$ and $-z\in \mathcal{N}_{\mathcal{C}(\ell, u)}(x)$ are equivalent to the following nonsmooth equation:
\[
x - \mathrm{Proj}_{\mathcal{C}(\ell, u)} (x - z) = 0,
\]
which implies that for any KKT point $(x,y,w,z)$ satisfying \eqref{eq-kkt-proj-F}, it holds that 
\begin{equation}
	\label{eq-w}
	w = \mathrm{Proj}_{\mathcal{C}(\ell, u)} (ye+\bar{x}),\quad z = \mathrm{Proj}_{\mathcal{C}(\ell, u)} (ye+\bar{x}) - (ye+\bar{x}).
\end{equation}

Substituting the above relations \eqref{eq-w} into the dual problem \eqref{eq-proj-F-dual}, we derive an equivalent optimization with only one decision variable:
\begin{equation}
	\label{eq-proj-F-dual-new}
	\min_{y\in\mathbb{R}}\; \phi(y):= \frac{1}{2}\|\mathrm{Proj}_{\mathcal{C}(\ell, u)} (ye+\bar{x})\|^2 - by + \delta_{\mathcal{C}(\ell, u)}^*(ye+\bar{x} -\mathrm{Proj}_{\mathcal{C}(\ell, u)} (ye+\bar{x})).
\end{equation}
It is clear that the objective function $\phi:\mathbb{R}\to\mathbb{R}$ is continuously differentiable and convex, thus solving the problem \eqref{eq-proj-F-dual-new} is equivalent to solving the following nonsmooth equation:
\begin{equation}
	\label{eq-proj-non-smooth}
	\phi'(y):= e^T\mathrm{Proj}_{\mathcal{C}(\ell, u)} (ye+\bar{x}) - b = 0,\quad y\in\mathbb{R}.
\end{equation}
Here, the continuous differentiability and the expression of the derivative of $\phi$ is based on the elegant properties of the Moreau envelope for a convex function; see e.g., \cite[Theorem 2.26]{rockafellar2009variational} for more details. After obtaining an optimal solution $\bar{y}$ of the above problem \eqref{eq-proj-F-dual-new}, the optimal solution of problem \eqref{eq-proj-F} can be constructed using 
\[
\mathrm{Proj}_{\mathcal{F}}(\bar{x}) = \mathrm{Proj}_{\mathcal{C}(\ell, u)} (\bar{y}e+\bar{x}).
\]
The remaining task in this section is to describe a Newton-type method for solving the nonsmooth equation \eqref{eq-proj-non-smooth} efficiently and accurately.

Though the derivative of $\phi$ is nonsmooth, one is able to show that it is strongly semismooth, which facilitates our design of efficient second-order algorithms, as we shall see shortly. Indeed, we see from Lemma \ref{lemma-semismooth-piecewise-linear} that the piecewise affine mapping $\mathrm{Proj}_{\mathcal{C}(\ell, u)}$ is strongly semismooth on $\mathbb{R}^n$. Hence, $\phi'$ defined in \eqref{eq-proj-non-smooth}, as the composition of strongly semismooth functions, is also strongly semismooth. To derive a second-order approximation of the function $\phi$, we then need to characterize the generalized Jacobian of $\phi'$, i.e., $\partial (\phi')$. Note that $\partial (\phi')$ is also known as the generalized Hessian of $\phi$ in the literature \cite{hiriart1984generalized}, hence, it is usually denoted as $\partial^2\phi$. Using $\partial^2\phi$, one is able to establish a ``second-order Taylor expansion'' of the continuously differentiable function $\phi$ with local Lipschitz continuous derivatives, as stated in the following lemma.

\begin{lemma}[{\cite{hiriart1984generalized}}]
	\label{lemma-taylor-expansion}
	For any $y$ and $\hat{y}$, it holds that 
	\[
	\phi(\hat{y}) = \phi(y) + \phi'(y) (\hat{y}-y) + \frac{1}{2} v  (\hat{y}-y)^2,
	\]
	where $v\in \partial^2\phi(y'')$, for some $y''$ lying between $y$ and $\hat{y}$. 
\end{lemma}

However, an explicit description of the set $\partial^2\phi(y)$ at a given $y$ is generally not available. Fortunately, \cite{hiriart1984generalized} shows that
\[
\partial^2\phi(y)  \Delta y \subseteq \hat{\partial}^2\phi(y)  \Delta y,\quad \forall\; \Delta y\in \mathbb{R},
\]
where $\hat{\partial}^2\phi(y)$ is defined as 
\[
\hat{\partial}^2\phi(y):=e^T\partial \mathrm{Proj}_{\mathcal{C}(\ell, u)}(ye+\bar{x})e\subseteq \mathbb{R}_+,\quad\forall\; y\in \mathbb{R},
\]
and a diagonal matrix $\mathcal{V} \in \mathbb{S}_+^n$ defined below is an element of $\partial \mathrm{Proj}_{\mathcal{C}(\ell, u)}(ye+\bar{x})$:
\begin{equation}
	\label{eq-Jacobian-proj-C}
	\mathcal{V}_{ij} := \left\{
	\begin{array}{ll}
		1, & \textrm{if } \ell_i \leq y + \bar{x}_i \leq u_i \textrm{\; and } i = j, \\
		0, & \textrm{otherwise.}
	\end{array}
	\right.
\end{equation}
Consequently, we can replace $v\in \partial^2\phi(y'')$ with $e^T\mathcal{V}e$ in the above lemma.

With the previous preparations, we can now present the semismooth Newton (SSN) method \cite{qi1993nonsmooth,li2020efficient} for solving the nonsmooth equation \eqref{eq-proj-non-smooth} in Algorithm \ref{alg-ssn}. In particular, Algorithm \ref{alg-ssn} together with its convergence properties is a direct application of the one proposed in \cite{li2020efficient}. However, to make the paper self-contained, we have included an analysis tailored specifically to our special settings; see also \cite{zhao2010newton,liang2021inexact}. 

\begin{algorithm}[htb!]
	\caption{A semismooth Newton method for problem \eqref{eq-proj-F-dual-new}.}
	\label{alg-ssn}
	\begin{algorithmic}[1]
		\STATE \textbf{Input:} $\bar{x}\in\mathbb{R}^n$, $b\in \mathbb{R}$, $\ell, u \in\mathbb{R}^n$ satisfying Assumption \ref{assumption-feasible-set}, $\delta\in(0,1)$, $\mu\in (0,1/2)$ and $\tau_1,\tau_2\in(0,1)$. 
		\STATE Choose $y^0\in\mathbb{R}$.
		\FOR{$k\geq 0$}
		\IF{$|\phi'(y)| = 0$}
		\STATE \textbf{Output: } $x^k$.
		\ENDIF
		
		\STATE Choose $\mathcal{V}^k\in\partial\mathrm{Proj}_{\mathcal{C}(\ell, u)}(y^ke + \bar{x})$ as in \eqref{eq-Jacobian-proj-C}. 
		
		\STATE Set $v^k:=e^T\mathcal{V}^ke$ and $\varepsilon^k:= \tau_1\min\{\tau_2, |\phi'(y^k)|\}$.
		
		\STATE Compute $\Delta y^k:= - \frac{1}{v^k+\varepsilon^k}\phi'(y^k)$.
		
		\STATE Set $\alpha^k:=\delta^{m_k}$, where $m_k$ is the smallest nonnegative integer $m$ for which 
		\[
		\phi(y^k + \delta^m \Delta y^k) \leq \phi(y^k) + \mu\delta^m \phi'(y^k)\Delta y^k.
		\]
		
		\STATE Set $y^{k+1} = y^k + \alpha^k \Delta y^k$.
		\ENDFOR 
	\end{algorithmic}
\end{algorithm}

The convergence properties of Algorithm \ref{alg-ssn} are summarized in the following theorem. Here again, we assume without loss of generality that Algorithm \ref{alg-ssn} generates an infinite sequence, denoted as $\{y^k\}_{k\geq 0}$. Otherwise, there exists $K$ such that $\phi'(y^K) = 0$. Then, due to the convexity of the function $\phi$, the output point $y^K$ must be an optimal solution for problem \eqref{eq-proj-F-dual-new}. 

\begin{theorem}
	\label{thm-convergence-ssn}
	The Algorithm \ref{alg-ssn} is well-defined. Let $\{y^k\}$ be an infinite sequence generated by Algorithm \ref{alg-ssn}. Then the whole sequence $\{y^k\}$ is bounded and any accumulation point of the sequence is an optimal solution to problem \eqref{eq-proj-F-dual-new}. Let $\bar{y}$ be an arbitrary accumulation point of $\{y^k\}$, assume that $\{i\in[n]\;:\; \ell_i < \bar{y} + \bar{x}_i < u_i\} \neq \emptyset $. Then, the whole sequence $\{y^k\}$ converges to $\bar{y}$ and 
	\begin{equation}
		\label{eq-thm-ssn-quad}
		|y^{k+1} - \bar{y}| = O(|y^k-\bar{y}|^2),\quad \textrm{as } k\to \infty. 
	\end{equation}
\end{theorem}
\begin{proof}
	See Appendix \ref{section-appendix}.
\end{proof}

\section{A vertex exchange method}\label{section-vem-qp}

In this section, we shall focus on solving the following strongly convex quadratic program (QP) problem subject to a generalized simplex constraint:
\begin{equation}
	\label{eq-qp}
	\min_{x\in\mathbb{R}^n}\; q(x):=\frac{1}{2}x^TQx+c^Tx \quad \mathrm{s.t.}\quad e^Tx = b,\; \ell\leq x\leq u, 
\end{equation}
where $Q\in\mathbb{S}_{++}^n$.  
Solving convex QPs efficiently, robustly and accurately has long been one of the most fundamental tasks in the field of optimization thanks to the numerous important application areas, including optimal control~\cite{qin2003survey}, operations research~\cite{hillier2015introduction}, and statistical and machine learning~\cite{james2013introduction,mohri2018foundations}, and its roles in developing more advanced algorithmic frameworks for general optimization tasks~\cite{pang1995globally,nocedal1999numerical,liu2022newton}. 

General-purpose convex QP solvers can address problem \eqref{eq-qp} directly, but they often become impractical as the problem size grows. Early methods, such as active set methods~\cite{wolfe1959simplex} that extend the simplex approach for linear programs~\cite{dantzig2016linear}, are expensive and may require iterations for large-scale problems. Interior point methods~\cite{nesterov1994interior,wright1997primal} offer strong theoretical guarantees and perform well for small to medium sizes but do not scale effectively to high dimensions. Augmented Lagrangian methods (ALMs)\cite{hermans2022qpalm,liang2022qppal} have gained traction for their strong empirical performance, yet repeatedly solving linear systems remains a bottleneck for very large problems. As a result, recent research has emphasized first-order methods, including Frank-Wolfe or conditional gradient methods\cite{frank1956algorithm}, projected gradient methods~\cite{nesterov1983method,beck2009fast}, and operator splitting methods~\cite{stellato2020osqp,o2016conic,li2016schur}. While Frank-Wolfe is appealing for its projection-free nature, its convergence can be slowed by the so-called “zigzagging” phenomenon. Renewed interest has led to variants such as the away-step Frank-Wolfe with linear convergence under strong convexity~\cite{lacoste2015global}, which proves effective for feasible regions given explicitly as convex hulls (e.g., the standard simplex~\cite{liu2022newton}). Projected gradient methods share similar convergence challenges, but combining them with active-set strategies yields more efficient schemes~\cite{hager2023algorithm}. Meanwhile, operator splitting methods often solve only one linear system in the preprocessing phase, yet this cost alone can dominate when problems are large.

Next, we shall present the vertex exchange method and its appealing convergence. The proposed algorithm relies on the duality and optimality associated with problem \eqref{eq-qp}. To see this, we first derive the associated Lagrange dual problem of \eqref{eq-qp} as follows (see e.g., \cite{liang2022qppal}):
\begin{equation}
    \label{eq-qp-dual}
    \max_{y\in\mathbb{R},\; w\in \mathbb{R}^n, \; z\in\mathbb{R}^n}\; by - \frac{1}{2}w^TQw- \mathcal{I}_{\mathcal{C}(\ell, u)}^*(-z)\quad \mathrm{s.t.}\quad -Qw + ye + z = c.
\end{equation}
Consequently, the first-order optimality conditions (also known as the KKT conditions) for problems \eqref{eq-qp} and \eqref{eq-qp-dual} are given as:
\begin{equation}
    \label{eq-kkt-qp}
    e^Tx = b,\quad -Qw + ye+z = c,\quad Qx-Qw = 0,\quad x\in\mathcal{C}(\ell, u),\quad -z\in \mathcal{N}_{\mathcal{C}(\ell, u)}(x).
\end{equation}
Since the Slater condition holds for problem \eqref{eq-qp} under Assumption \ref{assumption-feasible-set}, we see from \cite[Theorem 2.165]{bonnans2013perturbation} that the optimality conditions \eqref{eq-kkt-qp} admit at least one solution. Hence, by using \eqref{eq-normalcone-C} and the fact that $Q$ is nonsingular, we see that the KKT conditions provided in \eqref{eq-kkt-qp} can be reformulated as 
\begin{equation}
    \label{eq-kkt-qp-new}
    \begin{aligned}
        &\; e^Tx = b,\quad x\in\mathcal{C}(\ell,u), \quad -Qx + ye+z = c,\\
        &\; z_i\geq 0,\; \forall\; i\in J_\ell(x),\quad  z_i\leq 0,\;\forall\; i\in J_u(x),\quad z_i = 0,\; \forall\; i \notin J_\ell(x)\cup J_u(x).
    \end{aligned}
\end{equation}
Here, $J_\ell(x)$ and $J_u(x)$ denote the index sets:
\begin{equation}
	\label{eq-J}
	J_\ell(x):=\{i\in[n]\;:\; x_i = \ell_i\},\quad J_u(x):=\{i\in[n]\;:\; x_i = u_i\},\quad \forall x\in \mathbb{R}^n.
\end{equation}

Using the reformulated optimality conditions \eqref{eq-kkt-qp-new} and {being motivated by \cite{bohning1986vertex}}, we are able to provide an equivalent characterization of these conditions. Moreover, when the new condition is not satisfied, it then offers a descent direction for the QP problem \eqref{eq-qp}. These results are summarized in the following lemma, which plays the essential role in developing and analyzing our main algorithm, as we shall see shortly. 

\begin{lemma}
    \label{lemma-qp-descend}
    Suppose that Assumption \ref{assumption-feasible-set} holds. Let $x\in \mathbb{R}^n$ be a feasible point for problem \eqref{eq-qp} and $g:= \nabla q(x) = c + Qx$ be the gradient of $q$ at $x$. Select 
    \[
       s\in \mathrm{argmax}\left\{g_i\;:\; x_i > \ell_i  \right\},\quad t\in \mathrm{argmin}\left\{g_i\;:\; x_i < u_i \right\}.
    \]
    Then, $x$ is an optimal solution of problem \eqref{eq-qp} if and only if $g_s\leq g_t$. 
    
    Define a new point $x^+\in\mathbb{R}^n$ via:
    \[
        x^+ := x + \eta (e_{t} -  e_{s}),\quad \eta\in [0, \min\{x_{s} - \ell_{s}, u_{t} - x_{t}\}].
    \]
    Then, $x^+$ is a feasible solution for problem \eqref{eq-qp}, and if $g_s > g_t$,  $e_{t} - e_{s}$ is a descent direction of the function $q$.
\end{lemma}
\begin{proof}
    See Appendix~\ref{proof-lemma-qp-descend}.
\end{proof}

Now, Lemma \ref{lemma-qp-descend} naturally leads to the main algorithm, as presented in the following Algorithm \ref{alg-qp}. 

\begin{algorithm}[htb!]
    \caption{A Vertex Exchange Method for problem \eqref{eq-qp}.}
    \label{alg-qp}
    \begin{algorithmic}[1]
        \STATE \textbf{Input:} $Q\in\mathbb{S}_{++}^n$, $c\in \mathbb{R}^n$, $ b \in \mathbb{R}$, $\ell, u \in\mathbb{R}^n$ satisfying Assumption \ref{assumption-feasible-set}, and $\tilde x^0\in\mathbb{R}^n$.
        
        \STATE Compute $x^0 = \mathrm{Proj}_{\mathcal{F}}(\tilde x^0)$. 
        
        \STATE Compute $g^0 = \nabla q(x^0) = c + Qx^0$.
        \FOR{$k\geq 0$}
            
            \STATE Select $s^k\in \mathrm{argmax}\left\{g^k_i\;:\; x_i^k > \ell_i   \right\}$ and $t^k \in \mathrm{argmin}\left\{g^k_i\;:\; x_i^k < u_i \right\}$.

            \IF{$g_{s^k}^k \leq  g_{t^k}^k$}
                \STATE \textbf{Output:} $x^k$.
            \ENDIF

            \STATE Set $d^k = e_{t^k} - e_{ s^k}$ and $\eta_{\rm max}^k:=\min\{x_{s^k}^k - \ell_{s^k}, u_{t^k} - x_{t^k}^k\}$.
            
            \STATE Set $x^{k+1} = x^k + \eta^kd^k$ where $\eta_k$ is the optimal solution of 
            \begin{equation}
                \label{eq-alg-step-size}
                \min_{\eta\in\mathbb{R}}\; q(x^k + \eta d^k)\quad \mathrm{s.t.}\quad 0\leq \eta \leq \eta_{\rm max}^k.
            \end{equation}
            
            \STATE Update $g^{k+1} = g^k + \eta^k(Q_{:, t^k} - Q_{:, s^k})$.
        \ENDFOR
    \end{algorithmic}
\end{algorithm}

We note that the termination condition $g_{s^k}^k \leq  g_{t^k}^k$ {holds only theoretically}. In fact, this condition may never be met if the algorithm is executed on a system using floating point arithmetic. Hence, in practice, one can use $|g_{s^k}^k -  g_{t^k}^k| < \texttt{tol}$ or a similar condition for termination, where $\texttt{tol}>0$ is a given small tolerance specified by the user. Additionally, the algorithm requires only a single projection and one matrix-vector multiplication {at the beginning of the algorithm}. Thus, when $Q$ is large-scale and fully dense, the per-iteration cost of the Algorithm \ref{alg-qp} is extremely low, as long as one can access any column of $Q$ efficiently. For the computation of the step size $\eta^k$, we need to solve an optimization problem in one variable subject to a box constraint. Interestingly, this optimization problem admits an analytical solution, which is described in the following lemma.

\begin{lemma}
    \label{lemma-step-size}
    The optimal step size for problem \eqref{eq-alg-step-size} is given by: 
    \[
        \eta^k:=\min\left\{\eta_{\rm max}^k, \frac{g^k_{s^k} - g^k_{t^k}}{Q_{s^k, s^k} + Q_{t^k, t^k} - Q_{s^k, t^k} - Q_{t^k, s^k}}\right\}.
    \]
    Moreover, $\eta^k>0$ before the termination of Algorithm \ref{alg-qp}. 
\end{lemma}
\begin{proof}
	See Appendix~\ref{proof-lemma-step-size}.
\end{proof}

Recall that at the $k$-th iteration of the classical block-coordinate descent (BCD) method~\cite{tseng2001convergence}, we first select an index set $\mathcal{I}_k\subset [n]$ and then solve the following optimization problem for computing the search direction:
\begin{equation} \label{eq-eta-bcd}
    \min_{\eta_i, i\in \mathcal{I}_k}\; q(x^k + \sum_{i\in\mathcal{I}_k}\eta_ie_i)\quad \mathrm{s.t.}\quad e^T(x^k + \sum_{i\in\mathcal{I}_k}\eta_ie_i) = b,\; \ell \leq x^k + \sum_{i\in\mathcal{I}_k}\eta_ie_i \leq u.
\end{equation}
Denote the optimal solution for the above optimization problem \eqref{eq-eta-bcd} as $\eta_i^k$, for $i\in \mathcal{I}_k$, we then update the decision variable as 
\[
    x^{k+1} = x^k + \sum_{i\in \mathcal{I}_k}\eta_i^ke_i,\quad k\geq 0.
\]
Note that solving the optimization problem \eqref{eq-eta-bcd} can also be challenging, especially when the number of selected indices, i.e., $|\mathcal{I}_k|$ is large. 
The proposed VEM can be viewed as a greedy version of the BCD method, where we cleverly choose $\mathcal{I}_k = \{s^k, t^k\}$ based on the gradient at the current iterate.
In this way, we are able to show that the optimal step size admits a closed-form expression, as shown in Lemma~\ref{lemma-step-size}. 

The same idea has been applied to solve support vector machine problems in the so called sequential minimal optimization method, as studied in~\cite{chen2024study}. By using this connection, we can state the global convergence and the local linear convergence rate of Algorithm \ref{alg-qp}. Here, we assume without loss of generality that Algorithm \ref{alg-qp} generates an infinite sequence, denoted as $\{x^k\}_{k\geq 0}$. Otherwise, there exists $K$ such that $g_{s^K}^K \leq  g_{t^K}^K$. Then, Lemma \ref{lemma-qp-descend} implies that the output point $x^K$ is the optimal solution for problem \eqref{eq-qp}. 
\begin{theorem}{\cite{chen2024study,list2009svm}}
    \label{thm-convergence}
    Let $\{x^k\}$ be an infinite sequence generated by Algorithm \ref{alg-qp}. Then the whole sequence converges to the unique optimal solution $x^*$ of problem \eqref{eq-qp} with 
    \[
        \max\{[\nabla q(x^*)]_i: x_i^* > \ell_i\} = \min\{[\nabla q(x^*)]_i: x_i^* < u_i\}.
    \]
    Moreover, the sequence $\{x^k\}$ converges to $x^*$ linearly, for $k$ sufficiently large.
\end{theorem}

The linear convergence rate was first established in \cite[Theorem 8]{chen2024study} for support vector machine under the the nondegeneracy condition. It guarantees that when the iterate $x^k$ is sufficiently close to the optimal solution $x^*$, the active bounded constraints are fixed and hence will not be selected for constructing $\mathcal{I}_k = \{s^k, t^k\}$. However, the nondegeneracy condition can be quite restricted and it is generally not verifiable. Later, such a strong condition is removed in \cite{list2009svm}.

\section{Numerical experiments}\label{section-exp}
This section is dedicated to a set of comprehensive numerical studies to evaluate the practical performance of the proposed algorithms compared with existing state-of-the-art counterparts.

\textbf{Termination conditions.} For the semismooth Newton (\texttt{SSN}) method proposed in Algorithm \ref{alg-ssn}, we terminate the algorithm if it computes a point $y\in\mathbb{R}$ such that $|\phi'(y)|\leq \texttt{grad\_tol}$, where $\texttt{grad\_tol} > 0$ is a given tolerance specified by the user. In our experiments in Section \ref{subsection-num-proj}, we always set $\texttt{grad\_tol} = 10^{-12}$. The maximal number of iterations allowed for the \texttt{SSN} method is set to be $50$. For the main Algorithm \ref{alg-qp}, denoted as \texttt{VEM}, we offer several options. Given a tolerance $\texttt{tol} > 0$, one could terminate the algorithm if one of the following conditions is met at the $k$-th iteration:
\begin{itemize}
    \item $|g_{s^k}^k -  g_{t^k}^k|/\max\{1, \|Q\|_F\} \leq \texttt{tol}$.
    \item The relative optimality residue 
    \[
        \frac{\|x^k - \mathrm{Proj}_{\mathcal{F}}(x^k - Qx^k -c)\|}{1+\|x^k\|} \leq \texttt{tol}.
    \]
    \item $\texttt{err}(x^k)\leq \texttt{tol}$ where $\texttt{err}:\mathbb{R}^n\to \mathbb{R}_{+}$ is any user-defined error function.
\end{itemize}
If the projection onto the feasible set $\mathcal{F}$ is time-consuming, then one is suggested to employ the first termination condition. Otherwise, the second condition is more robust and commonly used in the literature as it directly measure the violation of the optimality conditions. Readers are referred to \eqref{eq-alg-pn-inexactness} for one particular example that employs the third condition. If not stated explicitly, we use the first termination condition as the default option with $\texttt{tol} = 10^{-12}$. Moreover, The maximum number of iterations allowed for the \texttt{VEM} method is set to be $10^{6}$.

\textbf{Computational platform.} {We performed the numerical experiments in Sections~\ref{subsection-num-proj} and \ref{subsection-random-qp} using MATLAB (R2023b) on a machine with 16GB of RAM and a 2.2GHz Quad-Core Intel Core i7 processor. To evaluate performance on larger-scale problems, we ran the experiments in Section~\ref{subsection-dopt} using MATLAB (R2023b) on a machine equipped with 64GB of RAM and a 3.4GHz 13th Generation Intel Core i7-13700K processor.}

\subsection{Projection onto the generalized simplex}\label{subsection-num-proj}
We first verify the practical performance of the \texttt{SSN} method, which is needed to provide a feasible point for the main Algorithm \ref{alg-qp}. To this end, we first describe how to generate  test problems of various sizes. Then, we introduce the state-of-the-art baseline solvers that we would like to compare with. Finally, we present the numerical results and provide detailed comparisons. 

\textbf{Testing instances.} We generate a set of large-scale random projection problems onto the feasible set $\mathcal{F}$ as follows (here, $n\in\mathbb{N}$ denotes the vector size) using {\sc Matlab}:

\begin{verbatim}
    rng("default"); % for reproducitivity
    l = max(0, randn(n,1)); 
    u = l + rand(n,1); 
    b = sum(l + u)/2;
    x0 = rand(n,1); % the point to project
\end{verbatim}
Here, $\texttt{x0}$ represents the point to be projected. In our experiments, we test $n\in \{10^{6}, 2\times 10^{6}, \dots, 10^{7}\}$.

\textbf{Baseline solvers.} The baseline solvers we choose to compare Algorithm \ref{alg-ssn} with are the commercial solver \texttt{GUROBI} (v11.0.1) \footnote{Solver website: \url{https://www.gurobi.com/}} with an academic license and a dual active set method that exploits sparsity of the problem, namely \texttt{PPROJ} \cite{hager2016projection} \footnote{Available at: \url{https://people.clas.ufl.edu/hager/software/}}. The \texttt{GUROBI} solver is currently one of the most efficient and robust solvers for solving general linear and quadratic programming problems, which is able to provide highly accurate solutions \footnote{See \url{https://plato.asu.edu/bench.html} for related benchmark.}. On the other hand, \texttt{PPROJ} makes special effort in identifying the solution sparsity to achieve superior practical performance, as indicated in \cite{hager2016projection}. The settings for \texttt{GUROBI} and \texttt{PPROJ} are summarized in the following Table \ref{table-settings-proj-gen-simplex}. Particularly, only the barrier method in \texttt{GUROBI} (which offers the best performance among the supported algorithms in \texttt{GUROBI}) is tested.

\begin{table}[htb!]
    \centering
    \begin{tabular}{c|c} \hline 
       Solver  & Settings\\ \hline 
       \texttt{GUROBI}  & \makecell{\texttt{Method = 2; BarConvTol = 1e-12;} \\ \texttt{FeasibilityTol = 1e-9; OptimalityTol = 1e-9;}}\\ \hline 
       \texttt{PPROJ}  & \texttt{grad\_tol = 1e-12;}\\ \hline 
    \end{tabular}
    \caption{Settings for the baseline solvers.}
    \label{table-settings-proj-gen-simplex}
\end{table}

The computational results are presented in Table \ref{table-proj-gen-simplex} in which we report the linear constraint violations and the total computational times (in seconds) for the three algorithms. Note that all solvers are able to return a point that belongs to the set $\mathcal{C}(\ell, u)$, based on our observation. Hence, for simplicity, we refrain from reporting the box constraint violations.

\begin{table}[htb!]
    \centering
    \begin{tabular}{c|ccc|ccc} \hline 
& \multicolumn{3}{c|}{$\displaystyle \| e^{T} x\ -b\| $} & \multicolumn{3}{c}{\texttt{Time}} \\ \hline 
 \texttt{Sizes} & \texttt{PPROJ} & \texttt{GUROBI} & \texttt{SSN} & \texttt{PPROJ} & \texttt{GUROBI} & \texttt{SSN} \\ \hline
\num{1e+6} & \num{1e-8} & \num{3e-8} & $<\texttt{eps}$ & 0.32 & 4.93  & 0.03 \\ 
\num{2e+6} & \num{3e-8} & \num{2e-8} & $<\texttt{eps}$ & 0.68 & 9.28  & 0.06 \\
\num{3e+6} & \num{2e-7} & \num{6e-8} & $<\texttt{eps}$ & 1.04 & 14.31 & 0.09 \\
\num{4e+6} & \num{2e-8} & \num{1e-6} & $<\texttt{eps}$ & 1.46 & 20.25 & 0.13 \\
\num{5e+6} & \num{2e-7} & \num{2e-9} & $<\texttt{eps}$ & 1.88 & 25.17 & 0.14 \\
\num{6e+6} & \num{2e-7} & \num{2e-8} & $<\texttt{eps}$ & 2.26 & 30.86 & 0.17 \\
\num{7e+6} & \num{5e-7} & \num{3e-7} & $<\texttt{eps}$ & 2.73 & 35.71 & 0.25 \\
\num{8e+6} & \num{3e-7} & \num{3e-7} & $<\texttt{eps}$ & 3.12 & 40.55 & 0.21 \\
\num{9e+6} & \num{5e-8} & \num{1e-6} & $<\texttt{eps}$ & 3.48 & 44.06 & 0.33 \\
\num{1e+7} & \num{3e-7} & \num{5e-7} & $<\texttt{eps}$ & 3.91 & 50.88 & 0.32 \\ 
\hline 
\end{tabular}
\captionof{table}{Linear constraint violations. Here the \texttt{eps} means the machine epsilon. For our computational platform, $\texttt{eps}:=2.2204\times 10^{-16}$.}
\label{table-proj-gen-simplex}
\end{table}

Obviously, the \texttt{SSN} method outperforms both \texttt{GUROBI} and $\texttt{PPROJ}$ in terms of both computational time and linear constraint violations. In fact, our computational results show that the \texttt{SSN} method is able to compute a projection point nearly exactly using much less computational time. $\texttt{PPROJ}$, which exploits the solution sparsity, also shows promising numerical performance and is able to generate solutions as accurate as those of \texttt{GUROBI}. \texttt{GUROBI}, on the other hand, is less efficiently due to the need of solving a linear system at each iteration, even though the linear system is highly sparse. We conclude that the \texttt{SSN} method for computing the projection onto the generalized simplex is highly efficient and accurate due to the fast convergence rate and can be used as a fundamental toolbox for other algorithms for which the projection is needed.

\subsection{QPs with generalized simplex constraints} \label{subsection-random-qp}

We next evaluate the numerical performance of the Algorithm \ref{alg-qp} via solving a set of artificial QP problems with known optimal solution and comparing it with current state-of-the-art first-order methods. To this end, we first describe how to generate the testing QP instances and then mention the baseline solvers to be compared with. Finally, we present the detailed computational results with some discussions. 

\textbf{Testing instances.} Given a matrix size $n \in \mathbb{N}$, we generate a set of symmetric positive definite matrices $Q\in\mathbb{S}_{++}^n$ with different condition numbers, denoted in short by \texttt{cond}. In particular, let $Q_0\in\mathbb{R}^{n\times n}$ be a random square matrix, we first compute the QR factorization of $Q_0$ to get an orthogonal matrix $U\in\mathbb{R}^{n\times n}$. Moreover, for a given \texttt{cond}, we {generate a random vector $d\in\mathbb{R}^n$ of integers from the closed interval $[1, \texttt{cond}]$ whose minimum and maximum entries are then scaled to 1 and $\texttt{cond}$, respectively. Then, the matrix $Q$ is set as $Q :=  \widetilde{Q} / \|\widetilde{Q}\|_F$ with $\widetilde{Q}:=U\mathrm{Diag}(d)U^T$, which is automatically positive definite with a condition number \texttt{cond}.} The optimal solution $\bar{x}\in\mathbb{R}^n$ is a random vector sampled uniformly from the interval $[-1, 1]$. Then, we set $b := e^T\bar{x}$. Let $\texttt{ratio}\in (0,1)$ be given, we can define two index sets $J_\ell:=\{i\in[n]\;:\; \bar{x}_i \leq -\texttt{ratio}\}$ and $J_u:=\{i\in[n]\;:\; \bar{x}_i \geq \texttt{ratio}\}$. We now construct the lower bound and upper bound vectors $\ell$ and $u$ as follows
\[
    \ell_i = \left\{
        \begin{array}{ll}
            \bar{x}_i, & i\in J_\ell, \\
            -1 & \textrm{otherwise},
        \end{array}
    \right.\quad  u_i = \left\{
        \begin{array}{ll}
            \bar{x}_i, & i\in J_u, \\
            1 & \textrm{otherwise},
        \end{array}
    \right.\quad \forall\; i\in[n].
\]
Given the generated $Q$, $\bar{x}$, $b$, $\ell$ and $u$, we next show how to generate $c\in\mathbb{R}^n$ by utilizing the first-order optimality conditions given in \eqref{eq-kkt-qp-new}. We propose to set $\bar{y}$ as a random number drawn from the standard normal distribution $\mathcal{N}(0,1)$. Furthermore, we construct a vector $\bar{z}\in\mathbb{R}^n$ as follows:
\[
    \bar{z}_i =\left\{
    \begin{array}{ll}
        \texttt{rand}(1), & i\in J_\ell, \\
         0, & i\notin J_\ell\cup J_u, \\
        -\texttt{rand}(1), & i\in J_\ell,
    \end{array} 
    \right.\quad \forall\; i\in [n],
\]
where $\texttt{rand}$ denotes the random number generator in {\sc Matlab} that samples a random number from the uniform distribution on the interval $[0,1]$. Then, the vector $c$ is constructed as {$c := -Q\bar{x}+\bar{y}e + \bar{z}$.} By construction, $(\bar{x}, \bar{y}, \bar{z})$ becomes a solution for the first-order optimality conditions \eqref{eq-kkt-qp-new}. In our experiment, we choose $n\in \{10^4, 2\times 10^4\}$, ${\texttt{ratio}\in \{0.2, 0.4, 0.6, 0.8\}}$, and $\texttt{cond}\in\{10^2, 10^4, 10^6, 10^8\}$. Consequently, the total number of test instances is $2\times 4 \times 4 = 32$. 

\textbf{Baseline solvers.} Note that the matrix $Q\in\mathbb{S}_{++}^n$ in our experiment is of large-scale and fully dense. Hence, an algorithm that requires solving linear systems would be inefficient. In view of this, we only compare the proposed algorithm with state-of-the-art first-order methods that do not need to solve linear systems \footnote{{We also tried \texttt{OSQP} \cite{stellato2020osqp}, {a specialized ADMM-based first-order method}, which presents as one of the state-of-the-art first-order methods for solving general convex quadratic programming problems. {Our numerical experience revealed that OSQP requires the factorization of a large symmetric positive definite matrix, with the time spent on this factorization dominating the overall computational cost, rendering it non-competitive for solving the QPs considered in this paper.} For simplicity, we do not report the computational results obtained from \texttt{OSQP}.}}. These methods include the following: 
\begin{itemize}
    \item {The classical projected gradient method \cite{goldstein1964convex} (denoted as \texttt{PG}) is a fundamental tool for many other advanced methods. As an representative example, the accelerated projected gradient method \cite{nesterov1983method,beck2009fast} (denoted as \texttt{FISTA}), which built upon \texttt{PG}, is a hugely popular method in the optimization community in the past decades. In our experiment, we use the implementations of \texttt{PG} and \texttt{FISTA} provided by \cite{beck2019fom}~\footnote{Available at: \url{https://www.tau.ac.il/~becka/home}.}. The termination tolerance for \texttt{PG} and \texttt{FISTA} is set to be $10^{-8}$, and the time limit is set as $300$ seconds. To allow both methods to gain their best performance, we use the economical versions in which function values are not computed. Moreover, the projection onto the feasible set at each iteration is computed using Algorithm \ref{alg-ssn}, and the exact Lipschitz constant with respect to $\nabla q(\cdot)$ needed for determining the step-size is given explicitly by the generated vector $d$. Note that in general, the estimation of the Lipschitz constant can also be time-consuming, since it involves computing the extreme eigenvalue of $Q$. Moreover, one may also apply a certain backtracking scheme to estimate the Lipschitz constant. However, this requires computing the objective value at each iteration, which could also be time-consuming, especially when $Q$ is fully dense and large-scale.}
    \item When dealing with polyhedral constraints, it is desirable to employ the active-set strategy to greatly improve the efficiency and reduce the dimensionality of the problem. In particular, we consider the solver called \texttt{PASA} developed in the package \texttt{SuiteOPT} \cite{hager2023algorithm}~\footnote{Available at: \url{https://people.clas.ufl.edu/hager/software/}.}, which implements a projected gradient-based polyhedral active set method that is shown to have excellent practical numerical performance. In our experiments, we adapt the default settings of \texttt{PASA} except that we set $\text{grad\_tol} = 10^{-12}$. Note that the projection onto the feasible set $\mathcal{F}$ in \texttt{PASA} is computed by the integrated routine \texttt{PPROJ}.
    \item Though computing a single projection onto the feasible set $\mathcal{F}$ can be done efficiently by using Algorithm \ref{alg-ssn}, projection-based algorithms, including \texttt{FISTA} and \texttt{PASA}, may require numerous iterations to converge. In this case, the cost taken for computing the projection can be a limiting factor on the practical efficiency. In view of this, we consider comparing with the standard projection-free Frank-Wolfe algorithm \cite{frank1956algorithm}, denoted as \texttt{FW}, which has gained tremendous popularity lately due to its scalability and appealing convergence properties. The step-size for \texttt{FW} at the $k$-th iteration is set as $1/(k+2)$, for $k\geq 0$ \footnote{Recent works (e.g., \cite{lacoste2015global}) have shown that when the feasible set of a strongly convex constrained optimization is the convex hull of a set of given points, the Away-Steps Frank-Wolfe algorithm admits a linear convergence rate which leads to promising numerical performance. However, for the feasible set $\mathcal{F}$ considered in the present paper, it is challenging to find a set of points, say $\mathcal{A}:=\{a_1,\dots, a_p\}$ such that the convex hull of $\mathcal{A}$ is $\mathcal{F}$, i.e., $\mathrm{conv}\{a_1,\dots, a_p\} = \mathcal{F}$. This explains why we only consider the standard Frank-Wolfe algorithm.}. We terminate the algorithm when $\|x^{k+1} - x^k\| \leq 10^{-3}$ or the computational time exceeds a 300-second time limit. For the choice of termination tolerance, we tried some smaller values, but the computational time will increase dramatically without significantly improving the quality of the solutions. {Finally, we mention that the linear program to be solved at each iteration of \texttt{FW} is computed efficiently by the simplex method implemented in \texttt{Gurobi}.} 
\end{itemize} 

{For all the solvers, we use the zeros vector in $\mathbb{R}^n$ as the common initial point.} The computational results for $n = 10^4$ and $n = 2\times 10^4$ are presented in Table \ref{table-qp-1e4} and Table \ref{table-qp-2e4}, respectively. In both tables, we report the relative error with respect to the known optimal solution $\bar{x}$, i.e., $\texttt{relerr}:=\|x - \bar{x}\| / (1+\|\bar{x}\|)$, where $x\in\mathbb{R}^n$ is the output of a method. Moreover, the total computational time (in seconds) taken by each solver is also recorded.

\begin{table}[htb!]
\centering 
\begin{tabular}{c|ccccc|ccccc} \hline 
 & \multicolumn{5}{c|}{\texttt{cond} = \num{1e2}} & \multicolumn{5}{c}{\texttt{cond} = \num{1e4}} \\   \hline
\texttt{ratio} & \texttt{PG} & \texttt{FISTA} & \texttt{PASA} & \texttt{FW} & \texttt{VEM} & \texttt{PG} & \texttt{FISTA} & \texttt{PASA} & \texttt{FW} & \texttt{VEM} \\ \hline 
\multirow{2}{*}{0.2} & \num{5e-10} & \num{3e-10} & \num{1e-11} & \num{2e-04} & \num{6e-12} & \num{4e-10} & \num{4e-10} & \num{3e-12} & \num{2e-04} & \num{3e-12} \\
 & 2.5 & 3.7 & 2.0 & 300.0 & 1.1 & 2.6 & 3.7 & 2.2 & 300.1 & 1.0 \\
 \multirow{2}{*}{0.4} & \num{9e-10} & \num{4e-10} & \num{4e-11} & \num{4e-04} & \num{9e-11} & \num{9e-10} & \num{4e-10} & \num{3e-11} & \num{4e-04} & \num{1e-10} \\
 & 4.4 & 6.3 & 3.5 & 300.0 & 1.8 & 4.6 & 6.5 & 3.3 & 300.1 & 1.9 \\
 \multirow{2}{*}{0.6} & \num{2e-09} & \num{5e-10} & \num{6e-11} & \num{5e-04} & \num{2e-10} & \num{2e-09} & \num{6e-10} & \num{8e-11} & \num{5e-04} & \num{2e-10} \\
 & 7.8 & 11.0 & 4.9 & 300.1 & 3.9 & 8.1 & 11.3 & 4.6 & 300.0 & 4.1 \\
 \multirow{2}{*}{0.8} & \num{4e-09} & \num{1e-09} & \num{1e-10} & \num{9e-04} & \num{8e-10} & \num{5e-09} & \num{1e-09} & \num{2e-10} & \num{1e-03} & \num{9e-10} \\
 & 15.8 & 21.1 & 8.2 & 300.0 & 10.6 & 19.6 & 26.2 & 8.4 & 300.0 & 11.7 \\
  \hline \hline 
 & \multicolumn{5}{c|}{\texttt{cond} = \num{1e6}} & \multicolumn{5}{c}{\texttt{cond} = \num{1e8}} \\ \hline 
\texttt{ratio} & \texttt{PG} & \texttt{FISTA} & \texttt{PASA} & \texttt{FW} & \texttt{VEM} & \texttt{PG} & \texttt{FISTA} & \texttt{PASA} & \texttt{FW} & \texttt{VEM} \\ \hline 
 \multirow{2}{*}{0.2} & \num{4e-10} & \num{4e-10} & \num{1e-11} & \num{2e-04} & \num{2e-11} & \num{4e-10} & \num{4e-10} & \num{2e-11} & \num{2e-04} & \num{7e-12} \\
 & 2.6 & 3.7 & 2.0 & 300.0 & 1.0 & 2.7 & 3.9 & 2.2 & 300.1 & 1.0 \\
 \multirow{2}{*}{0.4} & \num{7e-10} & \num{5e-10} & \num{4e-11} & \num{4e-04} & \num{9e-11} & \num{9e-10} & \num{3e-10} & \num{3e-11} & \num{4e-04} & \num{1e-10} \\
 & 4.7 & 6.5 & 2.9 & 300.0 & 1.9 & 4.2 & 6.0 & 2.9 & 300.1 & 1.8 \\
 \multirow{2}{*}{0.6} & \num{2e-09} & \num{5e-10} & \num{8e-11} & \num{5e-04} & \num{3e-10} & \num{2e-09} & \num{5e-10} & \num{4e-11} & \num{6e-04} & \num{3e-10} \\
 & 8.0 & 11.1 & 4.9 & 300.0 & 3.8 & 7.8 & 10.8 & 5.1 & 300.0 & 3.9 \\
 \multirow{2}{*}{0.8} & \num{5e-09} & \num{1e-09} & \num{2e-10} & \num{1e-03} & \num{9e-10} & \num{5e-09} & \num{8e-10} & \num{2e-10} & \num{1e-03} & \num{9e-10} \\
 & 21.0 & 27.9 & 8.6 & 300.0 & 12.0 & 19.3 & 25.9 & 9.0 & 300.0 & 11.7 \\
  \hline  
\end{tabular}
\caption{Computational results for $n = 10^4$. For each $\texttt{ratio}\in\{0.2, 0.4, 0.6, 0.8\}$, the first row represents the relative error with respect to the true solution, and the second row records the computational time taken by each solver.}
\label{table-qp-1e4}
\end{table}

\begin{table}[htb!]
\centering 
\begin{tabular}{c|ccccc|ccccc} \hline 
 & \multicolumn{5}{c|}{\texttt{cond} = \num{1e2}} & \multicolumn{5}{c}{\texttt{cond} = \num{1e4}} \\   \hline
\texttt{ratio} & \texttt{PG} & \texttt{FISTA} & \texttt{PASA} & \texttt{FW} & \texttt{VEM} & \texttt{PG} & \texttt{FISTA} & \texttt{PASA} & \texttt{FW} & \texttt{VEM} \\ \hline 
 \multirow{2}{*}{0.2} & \num{3e-10} & \num{3e-10} & \num{1e-13} & \num{7e-04} & \num{5e-11} & \num{3e-10} & \num{2e-10} & \num{4e-11} & \num{7e-04} & \num{6e-11} \\
 & 10.4 & 14.8 & 20.3 & 300.1 & 3.1 & 10.4 & 16.0 & 18.4 & 300.2 & 3.1 \\
 \multirow{2}{*}{0.4} & \num{6e-10} & \num{3e-10} & \num{4e-11} & \num{1e-03} & \num{1e-10} & \num{6e-10} & \num{2e-10} & \num{4e-11} & \num{1e-03} & \num{2e-10} \\
 & 16.4 & 23.4 & 23.3 & 300.1 & 5.7 & 16.9 & 24.4 & 23.3 & 300.1 & 5.9 \\
 \multirow{2}{*}{0.6} & \num{1e-09} & \num{5e-10} & \num{1e-10} & \num{2e-03} & \num{4e-10} & \num{1e-09} & \num{4e-10} & \num{9e-11} & \num{2e-03} & \num{4e-10} \\
 & 30.1 & 41.9 & 29.7 & 300.2 & 12.5 & 31.8 & 44.8 & 20.4 & 300.2 & 12.8 \\
 \multirow{2}{*}{0.8} & \num{3e-09} & \num{5e-10} & \num{2e-10} & \num{3e-03} & \num{1e-09} & \num{3e-09} & \num{8e-10} & \num{2e-10} & \num{3e-03} & \num{1e-09} \\
 & 64.2 & 87.9 & 44.5 & 300.0 & 34.1 & 76.1 & 101.8 & 32.9 & 300.1 & 38.7 \\
  \hline \hline 
 & \multicolumn{5}{c|}{\texttt{cond} = \num{1e6}} & \multicolumn{5}{c}{\texttt{cond} = \num{1e8}} \\ \hline 
\texttt{ratio} & \texttt{PG} & \texttt{FISTA} & \texttt{PASA} & \texttt{FW} & \texttt{VEM} & \texttt{PG} & \texttt{FISTA} & \texttt{PASA} & \texttt{FW} & \texttt{VEM} \\ \hline 
 \multirow{2}{*}{0.2} & \num{3e-10} & \num{3e-10} & \num{1e-13} & \num{7e-04} & \num{5e-11} & \num{3e-10} & \num{2e-10} & \num{4e-11} & \num{7e-04} & \num{6e-11} \\
 & 10.4 & 14.8 & 20.3 & 300.1 & 3.1 & 10.4 & 16.0 & 18.4 & 300.2 & 3.1 \\
 \multirow{2}{*}{0.4} & \num{6e-10} & \num{3e-10} & \num{4e-11} & \num{1e-03} & \num{1e-10} & \num{6e-10} & \num{2e-10} & \num{4e-11} & \num{1e-03} & \num{2e-10} \\
 & 16.4 & 23.4 & 23.3 & 300.1 & 5.7 & 16.9 & 24.4 & 23.3 & 300.1 & 5.9 \\
 \multirow{2}{*}{0.6} & \num{1e-09} & \num{5e-10} & \num{1e-10} & \num{2e-03} & \num{4e-10} & \num{1e-09} & \num{4e-10} & \num{9e-11} & \num{2e-03} & \num{4e-10} \\
 & 30.1 & 41.9 & 29.7 & 300.2 & 12.5 & 31.8 & 44.8 & 20.4 & 300.2 & 12.8 \\
 \multirow{2}{*}{0.8} & \num{3e-09} & \num{5e-10} & \num{2e-10} & \num{3e-03} & \num{1e-09} & \num{3e-09} & \num{8e-10} & \num{2e-10} & \num{3e-03} & \num{1e-09} \\
 & 64.2 & 87.9 & 44.5 & 300.0 & 34.1 & 76.1 & 101.8 & 32.9 & 300.1 & 38.7 \\
  \hline  
\end{tabular}
\caption{Computational results for $n = 2\times 10^4$. For each $\texttt{ratio}\in\{0.2, 0.4, 0.6, 0.8\}$, the first row represents the relative errors with respect to the true solution, and the second row records the computational time taken by each solver.}
\label{table-qp-2e4}
\end{table}


We observe from the computational results that \texttt{PG}, \texttt{FISTA}, \texttt{PASA} and \texttt{VEM} share a similar performance in terms of the accuracy of the returned solutions. This suggests that all the four methods are robust for computing high quality solutions with respect to the number of active constraints and the condition number of $Q$. On the other hand, \texttt{FW} shows poor performance in terms of the quality of the solution due to the slow convergence speed. This is further reflected by the recorded computational time. Indeed, \texttt{FW} consumes much more computational time than those of the remaining methods and always reaches the time limit of 300 seconds. We note here that the slow convergence of \texttt{FW} is mainly due to its zigzagging phenomenon--a well-known phenomenon in the literature. We also observe that by exploiting the problem structure via identifying the active-set in the projected gradient-based methods, the efficiency can be greatly improved. This explains why \texttt{PASA} is more efficient than \texttt{PG} and \texttt{FISTA}. For the comparison between \texttt{PG} and \texttt{FISTA}, we observe that \texttt{PG} is slightly better than \texttt{FISTA} in terms of computational efficiency. This is mainly due to the fact \texttt{PG} typically converges faster than \texttt{FISTA} at the early stage given the same initial condition. In our numerical experiments, we observe that both methods only require a few iterations to converge. For \texttt{VEM}, we see that it is highly efficient when $\texttt{ratio}$ is small, say $\texttt{ratio} \leq 0.6$, and is less efficient when $\texttt{ratio}$ is larger (but still comparable to \texttt{PASA}). This suggests that \texttt{VEM} is highly suitable for problems having optimal solutions with many active constraints. 

\subsection{D-optimal experimental design}\label{subsection-dopt} 

In this subsection, we consider solving the D-optimal experimental design that is a methodology commonly used in statistics and engineering to optimize the efficiency of experimental designs:
\begin{equation}
    \label{eq-dopt-design}
    \min_{x\in \mathbb{R}^n}\; f(x):= -\log\det(A\mathrm{Diag}(x)A^T)\quad \mathrm{s.t.}\quad e^Tx = 1, \; 0\leq x\leq 1,
\end{equation}
where $A = \begin{bmatrix}
    a_1 & \dots & a_n
\end{bmatrix} \in \mathbb{R}^{p\times n}$ represents the features (characterized by vectors in $\mathbb{R}^p$) of $n$ experiments. Note that the dual problem of \eqref{eq-dopt-design} can be interpreted as the problem of finding the ellipsoid with the minimum volume that covers a given set of points $\{a_1,\dots, a_n\}$, which is well known as the minimum-volume enclosing ellipsoid problem \cite{damla2008linear}. Developing efficient algorithmic frameworks for solving problem \eqref{eq-dopt-design} and/or its dual problem has been actively researched for decades; see, e.g., \cite{silvey1978algorithm,fedorov2000design,kumar2005minimum,pukelsheim2006optimal,damla2008linear,lu2013computing,liu2022newton,zhao2023analysis} and references therein. 

The gradient and the Hessian of $f$ are given as (see, e.g., \cite{lu2018relatively}):
\[
    \begin{aligned}
        \nabla f(x) = &\; -\mathrm{diag}\left(A^T(A\mathrm{Diag}(x)A^T)^{-1}A\right), \\
        \nabla^2 f(x) = &\; \left(A^T(A\mathrm{Diag}(x)A^T)^{-1}A\right) \circ \left(A^T(A\mathrm{Diag}(x)A^T)^{-1}A\right) ,
    \end{aligned}
\]
for any $x\in\mathrm{dom}(f)$. Here, ``$\circ$'' denotes the Hadamard product. Moreover, it is clear that the function $f$ is self-concordant \cite{sun2019generalized}. Clearly, evaluating the gradient and the Hessian matrix of $f(\cdot)$ can be time-consuming, especially when $n$ and $p$ are large. Moreover, an algorithm that requires solving linear systems at each iteration is not efficient for solving \eqref{eq-dopt-design}. Indeed, \cite{liu2022newton} recently demonstrates that by developing algorithmic frameworks that reduce the number of gradient and Hessian evaluations without requiring solving linear systems is a highly efficient approach for solving \eqref{eq-dopt-design}. In particular,  \cite{liu2022newton} applies the projected Newton method (\texttt{FWPN}) whose subproblem (a strongly convex QP problem with the standard simplex constraint) is solved by the Away-Steps Frank-Wolfe (\texttt{ASFW}) method \cite{lacoste2015global} that admits analytical expressions for calculating the step sizes  \cite{khachiyan1996rounding} \footnote{The implementation of the algorithm is available at: \url{https://github.com/unc-optimization/FWPN}.}. \texttt{ASFW} is highly efficient for solving strongly convex optimization problems with the standard simplex constraint, since it has a global linear convergence guarantee. 

In our experiments, we compare the practical performance of \texttt{NewVEM} with \texttt{FWPN}. The stopping condition for both methods is the same and specified by \eqref{eq-alg-pn-inexactness}.  We report the total computational time (in seconds), denoted as \texttt{TTime}, taken by the Algorithm \ref{alg-pn} with the QP subproblems being solved by \texttt{ASFW} and \texttt{VEM}, respectively. Furthermore, we also report the computational time taken for solving the QP subproblems (in seconds), denoted as \texttt{QPTime}, by the two methods. Lastly, the terminating $\lambda^k$ for each testing instance is also reported. In our experiment, we terminate the Algorithm \ref{alg-pn} when $\lambda^k \leq 10^{-3}$, as also used in \cite{liu2022newton}.

{
We first consider randomly generated data. Specifically, the testing data set is generated in exactly the same way as \cite{damla2008linear,liu2022newton}, i.e., the data points $\{a_1,\dots, a_n\}\subseteq \mathbb{R}^p$ are generated using an independent multinomial Gaussian distribution. The feasible initial point for Algorithm \ref{alg-pn} is chosen as $x^0 = \frac{1}{n}e$. Note that $p$ should be less than $n$ in order to ensure that the problem \eqref{eq-dopt-design} is well-defined. Table \ref{table-dopt-design} presents the computational results with $n\in\{10^4, 2\times 10^4,\dots, 5\times 10^4\}$ and $p = \frac{1}{10}n$. We note that in practical applications, $p \ll n$. Next we consider real data sets that are commonly used in the literature. Particularly, we consider the following three design spaces:
\begin{align*}
    \chi_1(n):= &\; \left\{a_i = (\exp(-\alpha_i), \alpha_i\exp(-\alpha_i), \exp(-2\alpha_i), \alpha_i\exp(-2\alpha_i)^T, i\in[n] \right\}, \\
    \chi_2(n):= &\; \left\{a_i = (1, \alpha_i, \alpha_i^2, \alpha_i^3)^T, i\in[n] \right\}, \\
    \chi_3(n):= &\; \left\{a_i = (\beta_i, \beta_i^2, \sin(2\pi\beta_i), \cos(2\pi\beta_i))^T, i\in[n] \right\}, 
\end{align*}
where $\alpha_i = \frac{3i}{n}$, and $\beta_i = \frac{i}{n}$, for $i\in[n]$. The space $\chi_1(n)$ is used to represent a compartmental model studied in \cite{atkinson1993optimum}, the space $\chi_2(n)$ refers to a polynomial regression model considered in \cite{yu2011d}, and the space $\chi_3(n)$ stands for the quadratic/trigonometric model in \cite{wynn1972results}. Tables \ref{table-dopt-design-1}--\ref{table-dopt-design-3} presents the computational results with $n\in\{10^4, 2\times 10^4,\dots, 5\times 10^4\}$ for each of the three design spaces. Note in these cases, $p = 4$. 
}

\begin{table}[htb!]
\centering
{
\begin{tabular}{c|ccc|ccc}\hline 
   \texttt{QP Solver}        & \multicolumn{3}{c|}{\texttt{FWPN}}            & \multicolumn{3}{c}{\texttt{NewVEM}}                               \\ \hline
$n$       & \texttt{TTime} & \texttt{QPTime} & $\lambda^k $                        & \texttt{TTime} & \texttt{QPTime} & $\lambda^k $                        \\ \hline 
 \num{1e+04} &  24.76 &   4.40 & \num{2.5e-04} &  13.05 &   2.14 & \num{5.7e-04} \\
 \num{2e+04} & 115.16 &  23.05 & \num{3.1e-04} &  34.13 &   2.32 & \num{7.7e-04} \\
 \num{3e+04} & 201.41 &  31.28 & \num{5.0e-04} &  93.62 &   6.51 & \num{4.7e-04} \\
 \num{4e+04} & 345.61 &  47.26 & \num{1.7e-06} & 125.43 &   7.44 & \num{8.8e-04} \\
 \num{5e+04} & 438.34 &  36.32 & \num{2.7e-05} & 195.45 &  11.19 & \num{6.3e-04} \\
\hline
\end{tabular}}
\caption{Computational results on D-optimal experiment design with randomly generated data.}
\label{table-dopt-design}
\end{table}

\begin{table}[htb!]
\centering
{
\begin{tabular}{c|ccc|ccc}\hline 
   \texttt{QP Solver}        & \multicolumn{3}{c|}{\texttt{FWPN}}            & \multicolumn{3}{c}{\texttt{NewVEM}}                               \\ \hline
$n$       & \texttt{TTime} & \texttt{QPTime} & $\lambda^k $                        & \texttt{TTime} & \texttt{QPTime} & $\lambda^k $                        \\ \hline 
 \num{1e+04} &  24.76 &   4.40 & \num{2.5e-04} &  13.05 &   2.14 & \num{5.7e-04} \\
 \num{2e+04} & 115.16 &  23.05 & \num{3.1e-04} &  34.13 &   2.32 & \num{7.7e-04} \\
 \num{3e+04} & 201.41 &  31.28 & \num{5.0e-04} &  93.62 &   6.51 & \num{4.7e-04} \\
 \num{4e+04} & 345.61 &  47.26 & \num{1.7e-06} & 125.43 &   7.44 & \num{8.8e-04} \\
 \num{5e+04} & 438.34 &  36.32 & \num{2.7e-05} & 195.45 &  11.19 & \num{6.3e-04} \\
\hline
\end{tabular}}
\caption{Computational results on D-optimal experiment design with design space $\chi_1(n)$.}
\label{table-dopt-design-1}
\end{table}

\begin{table}[htb!]
\centering
{
\begin{tabular}{c|ccc|ccc}\hline 
   \texttt{QP Solver}        & \multicolumn{3}{c|}{\texttt{FWPN}}            & \multicolumn{3}{c}{\texttt{NewVEM}}                               \\ \hline
$n$       & \texttt{TTime} & \texttt{QPTime} & $\lambda^k $                        & \texttt{TTime} & \texttt{QPTime} & $\lambda^k $                        \\ \hline 
 \num{1e+04} &  19.25 &   3.17 & \num{3.3e-04} &  10.76 &   2.01 & \num{4.1e-04} \\
 \num{2e+04} &  58.69 &  10.45 & \num{4.6e-04} &  36.96 &   3.58 & \num{5.8e-06} \\
 \num{3e+04} & 138.57 &  22.71 & \num{9.6e-04} &  78.81 &   6.38 & \num{6.5e-06} \\
 \num{4e+04} & 171.78 &  20.49 & \num{6.1e-04} & 130.85 &  10.06 & \num{1.9e-04} \\
 \num{5e+04} & 337.56 &  55.64 & \num{5.0e-04} & 197.19 &  11.68 & \num{7.9e-06} \\
\hline
\end{tabular}}
\caption{Computational results on D-optimal experiment design with design space $\chi_2(n)$.}
\label{table-dopt-design-2}
\end{table}

\begin{table}[htb!]
\centering
{
\begin{tabular}{c|ccc|ccc}\hline 
   \texttt{QP Solver}        & \multicolumn{3}{c|}{\texttt{FWPN}}            & \multicolumn{3}{c}{\texttt{NewVEM}}                               \\ \hline
$n$       & \texttt{TTime} & \texttt{QPTime} & $\lambda^k $                        & \texttt{TTime} & \texttt{QPTime} & $\lambda^k $                        \\ \hline 
 \num{1e+04} &  54.45 &  17.25 & \num{1.0e-04} &   7.81 &   1.27 & \num{6.7e-04} \\
 \num{2e+04} &  54.83 &   8.28 & \num{6.9e-04} &  27.12 &   2.34 & \num{4.5e-04} \\
 \num{3e+04} & 206.33 &  38.70 & \num{6.7e-06} &  59.05 &   4.24 & \num{3.5e-04} \\
 \num{4e+04} & 171.37 &  20.36 & \num{5.7e-04} & 101.86 &   8.27 & \num{9.7e-04} \\
 \num{5e+04} & 240.29 &  26.95 & \num{6.7e-04} & 156.19 &  12.99 & \num{8.7e-04} \\
\hline
\end{tabular}}
\caption{Computational results on D-optimal experiment design with design space $\chi_3(n)$.}
\label{table-dopt-design-3}
\end{table}

We see from Tables \ref{table-dopt-design}--\ref{table-dopt-design-3} that Algorithm \ref{alg-pn} solves all the problems successfully to the desired accuracy by using either \texttt{FWPN} or \texttt{NewVEM}. For the comparison of the computational time, we see that \texttt{NewVEM} is much more efficient than \texttt{FWPN} based on the recorded \texttt{QPTime}. The efficiency gained in solving the QP subproblems further reduces the total computational time for solving the problem \eqref{eq-dopt-design} to the desired accuracy, as indicated by the recorded \texttt{TTime}. Careful readers may also observe that the reduction in the total computational time is more than the reduction in computational time for solving the QP subproblems via \texttt{VEM}. The reason for this scenario is that \texttt{VEM} tends to return sparser solutions, which further reduces the cost in computing the gradient and the Hessian of the function $f$. However, another bottleneck of applying  Algorithm \ref{alg-pn} for solving problem \eqref{eq-dopt-design} is in the evaluation of $\nabla f(\cdot)$ and $\nabla^2 f(\cdot)$. We leave it as a further research topic on how to efficiently implement Algorithm \ref{alg-pn} such that the cost for evaluating the gradient and Hessian of $f$ is minimized. Lastly, we want to highlight that \texttt{NewVEM} is able to handle problems with the generalized simplex constraint due to the same reason as mentioned in Section \ref{subsection-random-qp}, which is more favorable than \texttt{FWPN} in practice.

\section{Conclusions and future research}\label{section-conclusions}

We have proposed an inexact projected Newton method combined with a vertex exchange method for solving the generalized simplex-constrained self-concordant minimization problem, which is fundamental in many real-world applications. The proposed double-loop framework achieves local linear convergence for both the outer and inner loops, ensuring high efficiency and suitability for the targeted optimization problem. Since the method requires a feasible starting point, we have also developed a highly efficient semismooth Newton method for computing projections onto the feasible set and established its convergence properties. Finally, extensive numerical experiments validate the strong practical performance of our approach, showing that it outperforms state-of-the-art solvers in both efficiency and accuracy.

Several future research directions are worth exploring. First, our results suggest that a similar framework can be extended to solve other important classes of constrained optimization problems. As an example, Appendix~\ref{section-sc1m} illustrates its application to generalized simplex-constrained $\text{SC}^1$ minimization. Further investigation is needed to evaluate the framework’s practical performance across a broader range of applications, including the exact optimal experimental design. Second, while the inexact projected Newton method is proven to have linear convergence, this contrasts with the typical expectation that Newton-type methods exhibit superlinear or even quadratic convergence. Closing this gap requires more advanced analytical tools. Last but not least, the linear convergence rate of the vertex exchange method depends on a nondegeneracy condition, which may be relatively restrictive. Investigating how to relax this assumption remains an important open question.


\section*{Acknowledgement}
Haizhao Yang was partially supported by the US National Science Foundation under awards DMS-2244988, DMS-2206333, the Office of Naval Research Award N00014-23-1-2007, and the DARPA D24AP00325-00.

{We thank the editor and the anonymous reviewer for taking care of our paper and providing numerous valuable suggestions, which have helped to significantly improve the quality of the paper.}

\section*{Conflict of interest}
The authors declare that they have no conflict of interest.

\appendix
\section{Proof of Theorem \ref{thm-convergence-ssn}} \label{section-appendix}

Since $v^k+\varepsilon^k >0$, $\phi'(y^k)\neq 0$, and 
    \begin{equation}
        \label{eq-thm-ssn-inprod}
        \phi'(y^k)\Delta y^k 
        = -\frac{1}{v^k+\varepsilon^k}(\phi'(y^k))^2  < 0,
    \end{equation}
    we see that $\Delta y^k$ is always a descent direction of $\phi$ for all $k\geq 0$, and hence, the Algorithm \ref{alg-ssn} is well-defined. Since the constraint matrix for problem \eqref{eq-proj-F} $e^T$ has full row rank, $\ell < u$, and $e^T\ell < b < e^Tu$, we see from \cite[Theorem 17' \& Theorem 18']{rockafellar1974conjugate} that the level set $\{y\in\mathbb{R}\;:\; \phi(y) \leq \phi(y^0)\}$ is a closed and bounded convex set. Thus, the fact that $\phi(y^k)\leq \phi(y^0)$ for all $k\geq 0$ implies that the whole sequence $\{y^k\}$ is bounded, and hence it admits at least one accumulation point. For the rest of the proof, we denote $\bar{y}$ as an arbitrary accumulation point of $\{y^k\}$, i.e., there exists an infinite sequence $\{k_j\}_{j\geq 0}$ such that $\lim_{j\to\infty}y^{k_j} = \bar{y}$. 
    
    We next prove that $\phi(\bar{y}) = 0$ by contradiction. To this end, we assume on the contrary that $\phi(\bar{y}) \neq 0$. By the continuity of $\phi'$ and the boundedness of the sequence $\{y^k\}$, we see that there exists positive constants $\gamma_1$ and $\gamma_2$ such that $\gamma_1 \leq |\phi'(y^{k_j})| \leq \gamma_2$, and there exists a positive constant $\bar{\varepsilon}$ such that $\varepsilon^{k_j}\geq \bar{\varepsilon}$, for $j\geq 0$ sufficiently large. Since $v^{k_j}\leq n$ and $\varepsilon^{k_j} \leq \tau_1\tau_2\leq 1$, we see form \eqref{eq-thm-ssn-inprod} that
    \begin{equation}
        \label{eq-thm-ssn-neg}
        \phi'(y^{k_j})\Delta y^{k_j} = -\frac{1}{v^{k_j}+\varepsilon^{k_j}}(\phi'(y^{k_j}))^2 \leq -\frac{1}{n+1}\gamma^2_1 < 0,
    \end{equation}
    and 
    \[
        |\Delta y^{k_j}| = \frac{1}{v^{k_j}+\varepsilon^{k_j}}|\phi'(y^{k_j})|\in \left[\frac{1}{n+1}\gamma_1, \frac{1}{\bar{\varepsilon}}\gamma_2\right],
    \]
    for $j \geq 0$ sufficiently large. Hence, by taking a subsequence if necessary, we may assume that $\lim_{j\to\infty}\Delta y^{k_j} = \Delta \bar{y}$, for some $\Delta\bar{y}\in\mathbb{R}$. As a result, \eqref{eq-thm-ssn-neg} implies that 
    \begin{equation}
        \label{eq-thm-ssn-neg-limit}
        \phi'(\bar{y})\Delta\bar{y} < 0.
    \end{equation}
    Moreover, it holds from \eqref{eq-thm-ssn-neg} that 
    \begin{equation}
        \label{eq-thm-ssn-sup}
        \limsup_{j\to\infty}\; \phi'(y^{k_j})\Delta y^{k_j} < 0.
    \end{equation}
    By the continuity of the function $\phi$ and the boundedness of the sequence $\{y^{k_j}\}$, we see that 
    \[
        \lim_{j\to\infty}\; \phi(y^{k_{j+1}}) - \phi(y^{k_j}) = 0.
    \]
    Now since the sequence $\{\phi(y^k)\}$ is non-increasing, we see that 
    \[
        \mu\delta^{m_{k_j}}\phi'(y^{k_j})\Delta y^{k_j} \geq \phi(y^{k_{j}+1}) - \phi(y^{k_j}) \geq \phi(y^{k_{j+1}}) - \phi(y^{k_j}),
    \]
    which implies that
    \begin{equation}
        \label{eq-thm-ssn-inf}
        \liminf_{j\to\infty}\; \delta^{m_{k_j}}\phi'(y^{k_j})\Delta y^{k_j} \geq 0.
    \end{equation}
    Consequently, \eqref{eq-thm-ssn-sup} and \eqref{eq-thm-ssn-inf} imply that $\limsup_{j\to\infty}\delta^{m_{k_j}} \leq 0$. Hence, it holds that 
    \begin{equation}
        \label{eq-thm-ssn-delta}
        \lim_{j\to\infty}\delta^{m_{k_j}} = 0.
    \end{equation}
    Then, for the line search scheme described in Line 10 of Algorithm \ref{alg-ssn}, we see that 
    \begin{equation}
        \label{eq-thm-ssn-ls}
        \phi(y^{k_j} + \delta^{m_{k_j}-1}\Delta y^{k_j}) - \phi(y^{k_j}) > \mu \delta^{m_{k_j}-1}\phi'(y^{k_j}) \Delta y^{k_j}.
    \end{equation}
    By dividing $\delta^{m_{k_j}-1}$ on the both side of \eqref{eq-thm-ssn-ls} and using \eqref{eq-thm-ssn-delta}, we derive that 
    \[
        \phi'(\bar{y})\Delta \bar{y} = \lim_{j\to\infty} \frac{\phi(y^{k_j} + \delta^{m_{k_j}-1}\Delta y^{k_j}) - \phi(y^{k_j})}{\delta^{m_{k_j}-1}} \geq \mu \phi'(\bar{y})\Delta \bar{y}, 
    \]
    which together with the fact that $\mu\in (0,1/2)$ shows that $ \phi'(\bar{y})\Delta \bar{y} \geq 0$, hence contradicting \eqref{eq-thm-ssn-neg-limit}. Therefore, $\phi'(\bar{y}) = 0$, and by the convexity of $\phi$, $\bar{y}$ must be an optimal solution to the problem \eqref{eq-proj-F-dual-new}.

    Next, we assume that $\{i\in[n]\;:\; \ell_i < \bar{y} + \bar{x}_i < u_i\} \neq \emptyset $ and prove \eqref{eq-thm-ssn-quad}. The latter condition indicates that there exists an open neighborhood of $\bar{y}$, denoted as $\mathcal{N}$, such that for any $y\in \mathcal{N}$, the set $\partial^2\phi(y)$ is uniformly bounded and $|\phi'(y)| < \tau_2$, and for $j\geq 0$ sufficiently large, $v^{k_j}\geq 1$. By the strong semismoothness of $\phi'$, we see that, for $j\geq 0$ sufficiently large,
    \[
        \phi'(y^{k_j}) - \phi'(\bar{y}) - v^{k_j}(y^{k_j} - \bar{y}) = O(|y^{k_j} - \bar{y}|^2).
    \]

    Notice that for $j\geq 0$ sufficiently large, $\varepsilon^{k_j} = \tau_1|\phi(y^{k_j})|$, and hence, it holds that 
    \begin{equation}
        \label{eq-thm-ssn-ykjp}
        \begin{aligned}
            &\; |y^{k_j} + \Delta y^{k_j} - \bar{y}|\\
            = &\; |y^{k_j} - \bar{y} - (v^{k_j}+\varepsilon^{k_j})^{-1}\phi'(y^{k_j})| \\
            = &\; (v^{k_j}+\varepsilon^{k_j})^{-1}|\phi'(y^{k_j}) - \phi'(\bar{y}) - (v^{k_j}+\varepsilon^{k_j})(y^{k_j} - \bar{y})| \\
            \leq &\; O(|\phi'(y^{k_j}) - \phi'(\bar{y}) - (v^{k_j}+\varepsilon^{k_j})(y^{k_j} - \bar{y})|) \\
            \leq &\; O(|\phi'(y^{k_j}) - \phi'(\bar{y}) - v^{k_j}(y^{k_j} - \bar{y})|) + {O(|\phi'(y^{k_j})|\,|y^{k_j} - \bar{y}|)} \\
            \leq &\; O(|y^{k_j} - \bar{y}|^2) + o(|y^{k_j} - \bar{y}|) \\
            = &\; o(|y^{k_j} - \bar{y}|).
        \end{aligned}
    \end{equation}

    By Lemma \ref{lemma-taylor-expansion}, we have that 
    \begin{equation}
        \label{eq-thm-ssn-taylor}
        \begin{aligned}
            \phi(y^{k_j}) = &\; \phi(\bar{y}) + \frac{1}{2} v_{(1)}^{k_j} (y^{k_j} - \bar{y})^2, \\
            \phi(y^{k_j} + \Delta y^{k_j}) = &\; \phi(\bar{y}) + \frac{1}{2} v_{(2)}^{k_j} (y^{k_j} + \Delta y^{k_j} - \bar{y})^2,
        \end{aligned}
    \end{equation}
    where $v_{(1)}^{k_j} \in \partial^2\phi(y_{(1)}^{k_j})$ and $v_{(2)}^{k_j} \in \partial^2\phi(y_{(2)}^{k_j})$ for some $y_{(1)}^{k_j}$ lying between $y^{k_j}$ and $\bar{y}$, and $y_{(2)}^{k_j}$ lying between $y^{k_j} + \Delta y^{k_j}$ and $\bar{y}$, respectively. Note from \cite[Lemma 3.1]{facchinei1995minimization} that \eqref{eq-thm-ssn-taylor} implies that 
    \begin{equation}
        \label{eq-thm-ssn-dy}
        |y^{k_j} + \Delta y^{k_j} - \bar{y}| = o(|\Delta y^{k_j}|),\quad \lim_{j\to\infty} \frac{|\Delta y^{k_j}|}{|y^{k_j} - \bar{y}|} = 1,
    \end{equation}
    for $j\geq 0$ sufficiently large. Therefore, for $j\geq 0$ sufficiently large, $y_{(1)}^{k_j}, y_{(2)}^{k_j}\in \mathcal{N}$, and hence, $v_{(1)}^{k_j},v_{(2)}^{k_j}$ are uniformly bounded. Then, \eqref{eq-thm-ssn-ykjp}, \eqref{eq-thm-ssn-taylor}, \eqref{eq-thm-ssn-dy} and the fact that $\phi'$ is strongly semismooth yield that
    \[
        \begin{aligned}
            &\; \phi(y^{k_j} + \Delta y^{k_j}) - \phi(y^{k_j}) - \frac{1}{2}\phi'(y^{k_j}) \Delta y^{k_j} \\
            = &\;  -\frac{1}{2} v_{(1)}^{k_j} (y^{k_j} - \bar{y})^2 - \frac{1}{2}\phi'(y^{k_j}) \Delta y^{k_j} + o(|\Delta y^{k_j}|^2) \\
            = &\; -\frac{1}{2} v_{(1)}^{k_j} (y^{k_j} + \Delta y^{k_j} - \bar{y})(y^{k_j} - \bar{y})  - \frac{1}{2} (\phi'(y^{k_j}) - v_{(1)}^{k_j} (y^{k_j} - \bar{y})) \Delta y^{k_j} + o(|\Delta y^{k_j}|^2) \\
            = &\; - \frac{1}{2} (\phi'(y^{k_j}) - v_{(1)}^{k_j} (y^{k_j} - \bar{y})) \Delta y^{k_j} + o(|\Delta y^{k_j}|^2) \\
            = &\;  - \frac{1}{2} (\phi'(y^{k_j}) - \phi'(\bar{y}) - v_{(1)}^{k_j} (y^{k_j} - \bar{y})) \Delta y^{k_j} + o(|\Delta y^{k_j}|^2) \\
            = &\; O(|y^{k_j} - \bar{y}|^2)o(|\Delta y^{k_j}|) + o(|\Delta y^{k_j}|^2) \\
            = &\; o(|\Delta y^{k_j}|^2),
        \end{aligned}
    \]
    for $j\geq 0$ sufficiently large. The above relation further shows that 
    \[
        \phi(y^{k_j} + \Delta y^{k_j}) - \phi(y^{k_j}) \leq \mu \phi'(y^{k_j}) \Delta y^{k_j},
    \]
    for $j\geq 0$ sufficiently large since $\mu\in (0,1/2)$. Hence, for $j\geq 0$ sufficiently large, we can take $m_{k_j} = 0$, i.e., $\alpha_{k_j} = 1$, which indicates that $y^{k_j+1} = y^{k_j} + \Delta y^{k_j}$, and $|y^{k_j+1} - \bar{y}| < \frac{1}{2}|y^{k_j} - \bar{y}|$ (due to \eqref{eq-thm-ssn-ykjp}). The last inequality shows that $y^{k_j+1}\in \mathcal{N}$ and $\lim_{j\to\infty}y^{k_j+1}= \bar{y}$. Repeating the above process on $\{y^{k_j+1}\}$, we see that $y^{k_j+2} = y^{k_j+1} + \Delta y^{k_j+1}$, for sufficiently large $j$. By mathematical induction, it holds that the whole sequence $\{y^k\}$ converges to $\bar{y}$, and for $k\geq 0$ sufficiently large, $\alpha^k = 1$, i.e., the unit step size is eventually achieved. Now, by applying \eqref{eq-thm-ssn-dy} to the whole sequence $\{y^k\}$, similar to \eqref{eq-thm-ssn-ykjp}, we see that, for $k\geq 0$ sufficiently large,
    \[
        |y^{k+1} - \bar{y}| \leq O(|\phi'(y^{k}) - \phi'(\bar{y}) - v^k(y^k - \bar{y})|) + O(|\phi'(y^k)||y^k - \bar{y}|) = O(|y^k - \bar{y}|^2),
    \]
    which proves \eqref{eq-thm-ssn-quad}. Therefore, the proof is completed.
    
\section{Proof of Lemma~\ref{lemma-qp-descend}} \label{proof-lemma-qp-descend}

By the conditions that $e^T\ell < b < e^Tu$ and $\ell < u$ as stated in Assumption \ref{assumption-feasible-set}, the feasible set of problem \eqref{eq-qp} is nonempty and the indices $s$ and $t$ are well-defined. 
First, suppose that $g_s \leq g_t$. Let us consider the following two cases. Case 1: {$\{i\;:\; \ell_i  < x_i < u_i  \} \neq \emptyset$; and 
	Case 2: $\{i\;:\; \ell_i  < x_i < u_i  \} = \emptyset$.}

For Case 1, we see that 
\[
\begin{aligned}
	g_s = &\; \max\left\{g_i\;:\; \ell_i  <  x_i \leq u_i \right\} 
	\geq  \max\left\{g_i\;:\; \ell_i  < x_i < u_i  \right\} \\
	\geq &\; \min\left\{g_i\;:\; \ell_i   < x_i < u_i    \right\} 
	\geq  \min\left\{g_i\;:\; \ell_i \leq x_i < u_i  \right\} 
	=  g_t.
\end{aligned}
\]
This, together with the assumption that $g_s \leq g_t$, implies that $g_s = g_t$ and all inequalities in the above are in fact equalities. In particular, there exists a constant $y\in\mathbb{R}$ such that 
\[
g_i = y,\quad \forall\; i\in \{i\;:\; \ell_i  \leq x_i \leq u_i   \} = [n].
\]
By setting $z_i = 0$ for $i\in[n]$, we can construct a vector $z\in \mathbb{R}^n$ and a real number $y\in\mathbb{R}$ such that the triple $(x,y,z)$ satisfies the KKT condition \eqref{eq-kkt-qp-new}, i.e., $(x, y, z)$ is a KKT point. Thus, $x$ is the optimal solution for the primal problem \eqref{eq-qp} and $(y,z)$ is an optimal solution for the dual problem \eqref{eq-qp-dual}. 

For Case 2, i.e., $\{i\;:\; \ell_i < x_i < u_i\} = \emptyset$, we know that 
\[
\begin{aligned}
	g_s = &\; \max \left\{g_i\;:\; x_i > \ell_i\right\} = \max \left\{g_i\;:\; i \in J_u(x)\right\}, \\
	g_t = &\; \min \left\{g_i\;:\; x_i < u_i\right\} = \min \left\{g_i\;:\; i \in J_\ell(x)\right\},
\end{aligned} 
\]
which, together with the assumption that $g_s \leq g_t$, implies that 
\[
\max \left\{g_i\;:\; i \in J_u(x)\right\} \leq \min \left\{g_i\;:\; i \in J_\ell(x)\right\}.
\]
Let $y\in [g_s, g_t]$ be an arbitrary real number, we see that for any $i\in J_u(x)$, there exists $z_i \leq 0$ such that $g_i = y + z_i$, and for any $i\in J_\ell(x)$, there exists $z_i \geq 0$ such that $g_i = y + z_i$. Consequently, we have constructed a KKT point $(x, y, z)$, and due to the same reason as in Case 1, $x$ is the optimal solution for the primal problem \eqref{eq-qp} and $(y,z)$ is an optimal solution for the dual problem \eqref{eq-qp-dual}.

Next, we assume that $x$ is an optimal solution of problem \eqref{alg-qp}. Then there exist $y\in \mathbb{R}$ and $z\in \mathbb{R}^n$ satisfying the KKT condition \eqref{eq-kkt-qp-new}. It holds that 
\[
g_s = \max\{y+z_i\;:\; i\notin J_\ell(x)\} \leq y,\quad g_t = \min\{y + z_i\;:\; i \notin J_u(x)\} \geq y.
\]
Hence, $g_s \leq g_t$. This proves the first statement of the lemma.

For the second part of the proof, we verify that $x^+$ is a feasible solution of problem \eqref{eq-qp}. It is clear that 
\[
e^Tx^+ = e^Tx + \eta -\eta = b,  
\]
and from the condition that $\eta\in [0, \min\{x_{s} - \ell_{s}, u_{t} - x_{t}\}]$, it holds that
\[
\begin{aligned}
	x^+_{i} = &\; x_i \in [\ell_i, u_i],\; \forall \;i\notin\{s, t\}, \\
	x^+_{s} = &\; x_{s} - \eta \in [x_{s} - (x_{s} - \ell_{s}), x_{s}]\subseteq [\ell_{s}, u_{s}], \\
	x^+_{t} = &\; x_{t} + \eta \in [x_{t}, x_{t} + u_{t} - x_{t} ] \subseteq [\ell_{t}, u_{t}].
\end{aligned}
\]
Therefore, $x^+$ is a feasible solution for problem \eqref{eq-qp}.

Finally, to verify that $e_{t} - e_{s}$ is a descent direction of the function $q(x):=\frac{1}{2}x^TQx + c^Tx$ when $g_s > g_t$, we only need to verify that $g^T (e_{t} - e_{s}) < 0$. This is straightforward since 
\[
g^T (e_{t} - e_{s}) = g_t -  g_s < 0 .
\]
Therefore the proof is completed.

\section{Proof of Lemma~\ref{lemma-step-size}} \label{proof-lemma-step-size}

Note that 
\[
\begin{aligned}
	q(x^k + \eta d^k) = &\; \frac{1}{2}(x^k + \eta d^k)^TQ(x^k + \eta d^k) + c^T(x^k + \eta d^k) \\
	= &\; \frac{\eta^2}{2}(d^k)^TQd^k  + \eta(c^Td^k + (x^k)^TQd^k) + q(x^k) \\
	= &\; \frac{Q_{s^k, s^k} + Q_{t^k, t^k} - Q_{s^k, t^k} - Q_{t^k, s^k}}{2} \eta^2 + (g^k_{t^k} - g^k_{s^k})\eta + q(x^k).
\end{aligned}
\]
Since $Q\in\mathbb{S}_{++}^n$ and {$g^k_{s^k} > g^k_{t^k}$} before the termination of Algorithm \ref{alg-qp}, we see that $Q_{s^k, s^k} + Q_{t^k, t^k} - Q_{s^k, t^k} - Q_{t^k, s^k} > 0$ and $g^k_{t^k} - g^k_{s^k} < 0$ (see Lemma \ref{lemma-qp-descend}). Hence, it is easy to check that the optimal step-size is given by 
\[
\eta^k:= \min\left\{\eta_{\rm max}^k, \frac{g^k_{s^k} - g^k_{t^k}}{Q_{s^k, s^k} + Q_{t^k, t^k} - Q_{s^k, t^k} - Q_{t^k, s^k}}\right\}.
\]
To show that $\eta^k>0$, it suffices to verify that $\eta_{\rm max}^k$ is positive which is obvious by the definitions of $s^k$ and $t^k$. Therefore, the proof is completed.

\section{Constrained convex $\textrm{SC}^1$ minimization}\label{section-sc1m}

In this section, we consider one important of objective functions $f$, namely the class of $\textrm{SC}^1$ functions whose formal definitions will be given shortly. We present an inexact Newton-type method, which are originated from \cite{pang1995globally,chen2010inexact}, for solving these problems. At each iteration of a Newton method, one needs to solve a strongly convex quadratic programming problem of the form in \eqref{eq-qp}. Thus, Algorithm \ref{alg-qp} and Algorithm \ref{alg-ssn} set up the backbones of the presented inexact Newton-type method.

Let $\mathcal{O}$ be an open subset of $\mathbb{R}^n$. A function $f:\mathcal{O}\to \mathbb{R}$ is said to be $\textrm{LC}^1$ if it is differentiable everywhere on $\mathcal{O}$ and its gradient $\nabla f:\mathcal{O}\to \mathbb{R}^n$ is locally Lipschitz. Moreover, $f$ is said to be $\textrm{SC}^1$ if $\nabla f$ is semismooth everywhere on $\mathcal{O}$; see Definition \ref{def-semi-smooth}. In this subsection, we consider the problem \eqref{eq-sc-min} with $f$ being $\textrm{SC}^1$ and $\mathcal{O}$ containing $\mathcal{F}$. 

There are many interesting problems that can be modeled as convex  $\textrm{SC}^1$ minimization problems \cite{pang1995globally}. In the same work, the authors proposed a globally convergent SQP Newton method whose subproblems need to be solved exactly. Moreover, in order to ensure the local fast convergence rate, a stringent BD-regularity assumption was required. Later, a globally and superlineraly convergent inexact SQP Newton method was proposed in \cite{chen2010inexact}, and the BD-regularity assumption was relaxed by to a more realistic one. Furthermore, the subproblems were only required to be solved inexactly. The detailed description of the inexact SQP Newton method is presented in the following Algorithm \ref{alg-sqp}.

\begin{algorithm}[htb!]
	\caption{An inexact SQP Newton method for constrained $\textrm{SC}^1$ minimization.}
	\label{alg-sqp}
	\begin{algorithmic}[1]
		\STATE \textbf{Input:} $\tilde{x}^0\in \mathbb{R}^n$, and  parameters $\mu\in (0,1/2)$ and $\gamma,\; \rho,\; \delta, \;\tau_1,\; \tau_2\; \in (0,1)$.
		\STATE Set $x^0 = \mathrm{Proj}_{\mathcal{F}}(\tilde{x}^0)$, $f^{\rm pre} = \|G(x^0)\|$.
		\FOR{$k\geq 0$}
		\IF{$\|G(x^k)\| = 0$}
		\STATE \textbf{Output:} $x^k$.
		\ENDIF
		
		\STATE Compute $\varepsilon^k:=\tau_1\min\{\tau_2, \|G(x^k)\|\}$ and $\rho^k:=\min\{\rho, \|G(x^k)\|\}$.
		
		\STATE Select $\mathcal{V}_k\in\partial_B\nabla f(x^k)$ and apply Algorithm \ref{alg-qp} {(using $\tilde{x}^k$ as the initial point)} for computing an approximate solution $\tilde{x}^{k+1}$ of the following QP:
		\[
		\min_{x\in\mathbb{R}^n}\; q_k(x):=\frac{1}{2}(x - x^k)^T(\mathcal{V}_k+\varepsilon^k I_n)(x - x^k) + \nabla f(x^k)^T (x-x^k)\quad \mathrm{s.t.} \quad x\in \mathcal{F},
		\]
		in the sense that $\tilde{x}^{k+1} \in \mathcal{F}$, $q_k(\tilde{x}^{k+1}) \leq 0$, and 
		\[
		\|\tilde{x}^{k+1} - \mathrm{Proj}_\mathcal{F}(\tilde{x}^{k+1} - \nabla q_k(\tilde{x}^{k+1}))\|\leq \rho^k\|G(x^k)\|.
		\]
		
		\IF{$\|G(\tilde{x}^{k+1})\| \leq \gamma f^{\rm pre}$}
		\STATE Set $x^{k+1} = \tilde{x}^{k+1}$, $f^{\rm pre} = \|G(x^{k+1})\|$.
		\ELSE
		\STATE Set $\Delta x^k:= \tilde{x}^{k+1} - x^k$ and $x^{k+1} = x^k + \alpha^k \Delta x^k$, where $\alpha^k:= \delta^{m_k} $ and $m_k$ is the smallest nonnegative integer $m$ for which 
		$$
		f(x^k + \delta^m\Delta x^k)\leq f(x^k) + \mu\delta^m\nabla f(x^k)^T\Delta x^k.
		$$
		\IF{$\|G(x^{k+1})\| \leq \gamma f^{\rm pre}$}
		\STATE Set  $f^{\rm pre} = \|G(x^{k+1})\|$.
		\ENDIF
		\ENDIF
		\ENDFOR 
	\end{algorithmic}
\end{algorithm}

The convergence properties of Algorithm \ref{alg-sqp} are summarized as follows; see \cite{chen2010inexact} for a detailed analysis. 

\begin{theorem}
	\label{theorem-convergence-sqp}
	Suppose that Algorithm \ref{alg-sqp} generates an infinite sequence $\{x^k\}$. Then, it holds that $\lim_{k\to\infty}\; f(x^k) = f(x^*)$ where $x^*\in\mathcal{F}$ is an optimal solution of problem \eqref{eq-sc-min}. Moreover, suppose that $x^*$ is an accumulation point of the sequence $\{x^k\}$ and $I_n - \mathcal{S}(I_n-\mathcal{V})$ is nonsingular for any $\mathcal{S}\in \partial_B\mathrm{Proj}_\mathcal{F}(x^* - \nabla f(x^*))$ and $\mathcal{V}\in \partial_B\nabla f(x^*)$. Then, the whole sequence $\{x^k\}$ converges to $x^*$ superlinearly, i.e., 
	\[
	\|x^{k+1} - x^k\| = o(\|x^k - x^*\|).
	\]
	Moreover, if $\nabla f$ is strongly semismooth, then the rate of convergence is quadratic, i.e., 
	\[
	\|x^{k+1} - x^k\| = O(\|x^k - x^*\|^2).
	\]
\end{theorem}

\bibliographystyle{spmpsci}      
\bibliography{references}   

\begin{thebibliography}{10}
\providecommand{\url}[1]{{#1}}
\providecommand{\urlprefix}{URL }
\expandafter\ifx\csname urlstyle\endcsname\relax
  \providecommand{\doi}[1]{DOI~\discretionary{}{}{}#1}\else
  \providecommand{\doi}{DOI~\discretionary{}{}{}\begingroup
  \urlstyle{rm}\Url}\fi

\bibitem{ahipacsaouglu2021branch}
Ahipa{\c{s}}ao{\u{g}}lu, S.D.: A branch-and-bound algorithm for the exact
  optimal experimental design problem.
\newblock Statistics and Computing \textbf{31}(5), 65 (2021)

\bibitem{algoet1988asymptotic}
Algoet, P.H., Cover, T.M.: Asymptotic optimality and asymptotic equipartition
  properties of log-optimum investment.
\newblock The Annals of Probability pp. 876--898 (1988)

\bibitem{atkinson1993optimum}
Atkinson, A.C., Chaloner, K., Herzberg, A.M., Juritz, J.: Optimum experimental
  designs for properties of a compartmental model.
\newblock Biometrics pp. 325--337 (1993)

\bibitem{banjac2019infeasibility}
Banjac, G., Goulart, P., Stellato, B., Boyd, S.: Infeasibility detection in the
  alternating direction method of multipliers for convex optimization.
\newblock Journal of Optimization Theory and Applications \textbf{183},
  490--519 (2019)

\bibitem{beck2019fom}
Beck, A., Guttmann-Beck, N.: {FOM}--a {\sc matlab} toolbox of first-order
  methods for solving convex optimization problems.
\newblock Optimization Methods and Software \textbf{34}(1), 172--193 (2019)

\bibitem{beck2009fast}
Beck, A., Teboulle, M.: A fast iterative shrinkage-thresholding algorithm for
  linear inverse problems.
\newblock SIAM journal on imaging sciences \textbf{2}(1), 183--202 (2009)

\bibitem{ben2001ordered}
Ben-Tal, A., Margalit, T., Nemirovski, A.: The ordered subsets mirror descent
  optimization method with applications to tomography.
\newblock SIAM Journal on Optimization \textbf{12}(1), 79--108 (2001)

\bibitem{bohning1986vertex}
B{\"o}hning, D.: A vertex-exchange-method in {D}-optimal design theory.
\newblock Metrika \textbf{33}(1), 337--347 (1986)

\bibitem{bonnans2013perturbation}
Bonnans, J.F., Shapiro, A.: Perturbation analysis of optimization problems.
\newblock Springer Science \& Business Media (2013)

\bibitem{boyd2004convex}
Boyd, S.P., Vandenberghe, L.: Convex optimization.
\newblock Cambridge university press (2004)

\bibitem{chen2024study}
Chen, P.H., Fan, R.E., Lin, C.J.: A study on smo-type decomposition methods for
  support vector machines.
\newblock IEEE transactions on neural networks \textbf{17}(4), 893--908 (2006)

\bibitem{chen2010inexact}
Chen, Y., Gao, Y., Liu, Y.J.: An inexact {SQP} {Newton} method for convex
  ${SC}^1$ minimization problems.
\newblock Journal of optimization theory and applications \textbf{146}, 33--49
  (2010)

\bibitem{clarke1990optimization}
Clarke, F.H.: Optimization and nonsmooth analysis.
\newblock SIAM (1990)

\bibitem{cornuejols2006optimization}
Cornuejols, G., T{\"u}t{\"u}nc{\"u}, R.: Optimization methods in finance,
  vol.~5.
\newblock Cambridge University Press (2006)

\bibitem{damla2008linear}
Damla~Ahipasaoglu, S., Sun, P., Todd, M.J.: Linear convergence of a modified
  {Frank--Wolfe} algorithm for computing minimum-volume enclosing ellipsoids.
\newblock Optimisation Methods and Software \textbf{23}(1), 5--19 (2008)

\bibitem{dantzig2016linear}
Dantzig, G.B.: Linear programming and extensions.
\newblock In: Linear programming and extensions. Princeton university press
  (2016)

\bibitem{facchinei1995minimization}
Facchinei, F.: Minimization of ${SC}^1$ functions and the {M}aratos effect.
\newblock Operations Research Letters \textbf{17}(3), 131--137 (1995)

\bibitem{facchinei2003finite}
Facchinei, F., Pang, J.S.: Finite-dimensional variational inequalities and
  complementarity problems.
\newblock Springer (2003)

\bibitem{fedorov2000design}
Fedorov, V., Lee, J.: Design of experiments in statistics.
\newblock In: Handbook of Semidefinite Programming: Theory, Algorithms, and
  Applications, pp. 511--532. Springer (2000)

\bibitem{frank1956algorithm}
Frank, M., Wolfe, P.: An algorithm for quadratic programming.
\newblock Naval research logistics quarterly \textbf{3}(1-2), 95--110 (1956)

\bibitem{goldstein1964convex}
Goldstein, A.A.: Convex programming in hilbert space  (1964)

\bibitem{grant2006disciplined}
Grant, M., Boyd, S., Ye, Y.: Disciplined convex programming.
\newblock Springer (2006)

\bibitem{gross2010quantum}
Gross, D., Liu, Y.K., Flammia, S.T., Becker, S., Eisert, J.: Quantum state
  tomography via compressed sensing.
\newblock Physical review letters \textbf{105}(15), 150401 (2010)

\bibitem{hager2016projection}
Hager, W.W., Zhang, H.: Projection onto a polyhedron that exploits sparsity.
\newblock SIAM Journal on Optimization \textbf{26}(3), 1773--1798 (2016)

\bibitem{hager2023algorithm}
Hager, W.W., Zhang, H.: {Algorithm} 1035: a gradient-based implementation of
  the polyhedral active set algorithm.
\newblock ACM Transactions on Mathematical Software \textbf{49}(2), 1--13
  (2023)

\bibitem{hendrych2023solving}
Hendrych, D., Besan{\c{c}}on, M., Pokutta, S.: Solving the optimal experiment
  design problem with mixed-integer convex methods.
\newblock arXiv preprint arXiv:2312.11200  (2023)

\bibitem{hermans2022qpalm}
Hermans, B., Themelis, A., Patrinos, P.: {QPALM:} a proximal augmented
  {L}agrangian method for nonconvex quadratic programs.
\newblock Mathematical Programming Computation \textbf{14}(3), 497--541 (2022)

\bibitem{hillier2015introduction}
Hillier, F.S., Lieberman, G.J.: Introduction to operations research.
\newblock McGraw-Hill (2015)

\bibitem{hiriart1984generalized}
Hiriart-Urruty, J.B., Strodiot, J.J., Nguyen, V.H.: Generalized {H}essian
  matrix and second-order optimality conditions for problems with ${C}^{1, 1}$
  data.
\newblock Applied mathematics and optimization \textbf{11}(1), 43--56 (1984)

\bibitem{james2013introduction}
James, G., Witten, D., Hastie, T., Tibshirani, R., et~al.: An introduction to
  statistical learning, vol. 112.
\newblock Springer (2013)

\bibitem{khachiyan2009generating}
Khachiyan, L., Boros, E., Borys, K., Gurvich, V., Elbassioni, K.: Generating
  all vertices of a polyhedron is hard.
\newblock Twentieth Anniversary Volume: Discrete \& Computational Geometry pp.
  1--17 (2009)

\bibitem{khachiyan1996rounding}
Khachiyan, L.G.: Rounding of polytopes in the real number model of computation.
\newblock Mathematics of Operations Research \textbf{21}(2), 307--320 (1996)

\bibitem{kumar2005minimum}
Kumar, P., Yildirim, E.A.: Minimum-volume enclosing ellipsoids and core sets.
\newblock Journal of Optimization Theory and applications \textbf{126}(1),
  1--21 (2005)

\bibitem{lacoste2015global}
Lacoste-Julien, S., Jaggi, M.: On the global linear convergence of
  {Frank-Wolfe} optimization variants.
\newblock Advances in neural information processing systems \textbf{28} (2015)

\bibitem{lawler1966branch}
Lawler, E.L., Wood, D.E.: Branch-and-bound methods: A survey.
\newblock Operations research \textbf{14}(4), 699--719 (1966)

\bibitem{li2016schur}
Li, X., Sun, D., Toh, K.C.: A {Schur} complement based semi-proximal admm for
  convex quadratic conic programming and extensions.
\newblock Mathematical Programming \textbf{155}(1), 333--373 (2016)

\bibitem{li2020efficient}
Li, X., Sun, D., Toh, K.C.: On the efficient computation of a generalized
  {J}acobian of the projector over the {B}irkhoff polytope.
\newblock Mathematical Programming \textbf{179}(1), 419--446 (2020)

\bibitem{liang2022qppal}
Liang, L., Li, X., Sun, D., Toh, K.C.: {QPPAL}: a two-phase proximal augmented
  {L}agrangian method for high-dimensional convex quadratic programming
  problems.
\newblock ACM Transactions on Mathematical Software (TOMS) \textbf{48}(3),
  1--27 (2022)

\bibitem{liang2021inexact}
Liang, L., Sun, D., Toh, K.C.: An inexact augmented {Lagrangian} method for
  second-order cone programming with applications.
\newblock SIAM Journal on Optimization \textbf{31}(3), 1748--1773 (2021)

\bibitem{list2009svm}
List, N., Simon, H.U.: {SVM}-optimization and steepest-descent line search.
\newblock In: Proceedings of the 22nd Annual Conference on Computational
  Learning Theory (2009)

\bibitem{liu2022newton}
Liu, D., Cevher, V., Tran-Dinh, Q.: A {N}ewton {F}rank--{W}olfe method for
  constrained self-concordant minimization.
\newblock Journal of Global Optimization pp. 1--27 (2022)

\bibitem{lu2018relatively}
Lu, H., Freund, R.M., Nesterov, Y.: Relatively smooth convex optimization by
  first-order methods, and applications.
\newblock SIAM Journal on Optimization \textbf{28}(1), 333--354 (2018)

\bibitem{lu2013computing}
Lu, Z., Pong, T.K.: Computing optimal experimental designs via interior point
  method.
\newblock SIAM Journal on Matrix Analysis and Applications \textbf{34}(4),
  1556--1580 (2013)

\bibitem{mifflin1977semismooth}
Mifflin, R.: Semismooth and semiconvex functions in constrained optimization.
\newblock SIAM Journal on Control and Optimization \textbf{15}(6), 959--972
  (1977)

\bibitem{mohri2018foundations}
Mohri, M., Rostamizadeh, A., Talwalkar, A.: Foundations of machine learning.
\newblock MIT press (2018)

\bibitem{nesterov1983method}
Nesterov, Y.: A method for unconstrained convex minimization problem with the
  rate of convergence ${O}(1/k^2)$.
\newblock In: Dokl. Akad. Nauk. SSSR, vol. 269, p. 543 (1983)

\bibitem{nesterov1994interior}
Nesterov, Y., Nemirovskii, A.: Interior-point polynomial algorithms in convex
  programming.
\newblock SIAM (1994)

\bibitem{nocedal1999numerical}
Nocedal, J., Wright, S.J.: Numerical optimization.
\newblock Springer (1999)

\bibitem{o2016conic}
O’donoghue, B., Chu, E., Parikh, N., Boyd, S.: Conic optimization via
  operator splitting and homogeneous self-dual embedding.
\newblock Journal of Optimization Theory and Applications \textbf{169},
  1042--1068 (2016)

\bibitem{pang1995globally}
Pang, J.S., Qi, L.: A globally convergent {Newton} method for convex ${SC}^1$
  minimization problems.
\newblock Journal of Optimization Theory and Applications \textbf{85}(3),
  633--648 (1995)

\bibitem{peyre2019computational}
Peyr{\'e}, G., Cuturi, M., et~al.: Computational optimal transport: With
  applications to data science.
\newblock Foundations and Trends{\textregistered} in Machine Learning
  \textbf{11}(5-6), 355--607 (2019)

\bibitem{pukelsheim2006optimal}
Pukelsheim, F.: Optimal design of experiments.
\newblock SIAM (2006)

\bibitem{qi1993convergence}
Qi, L.: Convergence analysis of some algorithms for solving nonsmooth
  equations.
\newblock Mathematics of operations research \textbf{18}(1), 227--244 (1993)

\bibitem{qi1993nonsmooth}
Qi, L., Sun, J.: A nonsmooth version of {N}ewton's method.
\newblock Mathematical programming \textbf{58}(1), 353--367 (1993)

\bibitem{qin2003survey}
Qin, S.J., Badgwell, T.A.: A survey of industrial model predictive control
  technology.
\newblock Control engineering practice \textbf{11}(7), 733--764 (2003)

\bibitem{rademacher1919partielle}
Rademacher, H.: {\"U}ber partielle und totale differenzierbarkeit von
  funktionen mehrerer variabeln und {\"u}ber die transformation der
  doppelintegrale.
\newblock Mathematische Annalen \textbf{79}(4), 340--359 (1919)

\bibitem{rockafellar1997convex}
Rockafellar, R.: Convex {Analysis}.
\newblock Princeton {Landmarks} in {Mathematics} and {Physics}. Princeton
  University Press (1997)

\bibitem{rockafellar1974conjugate}
Rockafellar, R.T.: Conjugate duality and optimization.
\newblock SIAM (1974)

\bibitem{rockafellar2009variational}
Rockafellar, R.T., Wets, R.J.B.: Variational analysis, vol. 317.
\newblock Springer Science \& Business Media (2009)

\bibitem{ryu2014stochastic}
Ryu, E.K., Boyd, S.: Stochastic proximal iteration: a non-asymptotic
  improvement upon stochastic gradient descent.
\newblock Author website, early draft  (2014)

\bibitem{shepp1982maximum}
Shepp, L.A., Vardi, Y.: Maximum likelihood reconstruction for emission
  tomography.
\newblock IEEE transactions on medical imaging \textbf{1}(2), 113--122 (1982)

\bibitem{silvey1978algorithm}
Silvey, S.D., Titterington, D., Torsney, B.: An algorithm for optimal designs
  on a design space.
\newblock Communications in Statistics-Theory and Methods \textbf{7}(14),
  1379--1389 (1978)

\bibitem{slater2013lagrange}
Slater, M.: Lagrange multipliers revisited.
\newblock In: Traces and emergence of nonlinear programming, pp. 293--306.
  Springer (2013)

\bibitem{stellato2020osqp}
Stellato, B., Banjac, G., Goulart, P., Bemporad, A., Boyd, S.: {OSQP}: An
  operator splitting solver for quadratic programs.
\newblock Mathematical Programming Computation \textbf{12}(4), 637--672 (2020)

\bibitem{sun2002semismooth}
Sun, D., Sun, J.: Semismooth matrix-valued functions.
\newblock Mathematics of Operations Research \textbf{27}(1), 150--169 (2002)

\bibitem{sun2019generalized}
Sun, T., Tran-Dinh, Q.: Generalized self-concordant functions: a recipe for
  {N}ewton-type methods.
\newblock Mathematical Programming \textbf{178}(1), 145--213 (2019)

\bibitem{tran2015composite}
Tran-Dinh, Q., Kyrillidis, A., Cevher, V.: Composite self-concordant
  minimization.
\newblock J. Mach. Learn. Res. \textbf{16}(1), 371--416 (2015)

\bibitem{tran2022new}
Tran-Dinh, Q., Liang, L., Toh, K.C.: A new homotopy proximal variable-metric
  framework for composite convex minimization.
\newblock Mathematics of Operations Research \textbf{47}(1), 508--539 (2022)

\bibitem{tseng2001convergence}
Tseng, P.: Convergence of a block coordinate descent method for
  nondifferentiable minimization.
\newblock Journal of optimization theory and applications \textbf{109},
  475--494 (2001)

\bibitem{wolfe1959simplex}
Wolfe, P.: The simplex method for quadratic programming.
\newblock Econometrica: Journal of the Econometric Society pp. 382--398 (1959)

\bibitem{wright1997primal}
Wright, S.J.: Primal-dual interior-point methods.
\newblock SIAM (1997)

\bibitem{wynn1972results}
Wynn, H.P.: Results in the theory and construction of d-optimum experimental
  designs.
\newblock Journal of the Royal Statistical Society Series B: Statistical
  Methodology \textbf{34}(2), 133--147 (1972)

\bibitem{yang2020efficient}
Yang, C., Fan, J., Wu, Z., Udell, M.: Efficient automl pipeline search with
  matrix and tensor factorization.
\newblock arXiv preprint arXiv:2006.04216  (2020)

\bibitem{yu2011d}
Yu, Y.: D-optimal designs via a cocktail algorithm.
\newblock Statistics and Computing \textbf{21}, 475--481 (2011)

\bibitem{zhao2023away}
Zhao, R.: An away-step {Frank-Wolfe} method for minimizing
  logarithmically-homogeneous barriers.
\newblock arXiv preprint arXiv:2305.17808  (2023)

\bibitem{zhao2023analysis}
Zhao, R., Freund, R.M.: Analysis of the {Frank--Wolfe} method for convex
  composite optimization involving a logarithmically-homogeneous barrier.
\newblock Mathematical programming \textbf{199}(1), 123--163 (2023)

\bibitem{zhao2010newton}
Zhao, X.Y., Sun, D., Toh, K.C.: A {Newton-CG} augmented {L}agrangian method for
  semidefinite programming.
\newblock SIAM Journal on Optimization \textbf{20}(4), 1737--1765 (2010)

\end{thebibliography}

\end{document}